\documentclass[12pt]{smfart}

\usepackage{amssymb,amsthm,amsmath,amsxtra}
\usepackage[all]{xy}
\usepackage{fullpage}
\usepackage{comment}
\usepackage{hyperref}
\usepackage{color}
\usepackage{graphicx}
\usepackage{booktabs}

\numberwithin{equation}{section}

\theoremstyle{plain}
\newtheorem{thm}[equation]{Theorem}
\newtheorem{prop}[equation]{Proposition}
\newtheorem{lem}[equation]{Lemma}

\newtheorem{ques}[equation]{Question}

\newtheorem*{cor*}{Corollary}
\newtheorem*{prob*}{Problem}
\newtheorem*{thm*}{Theorem}
\newtheorem*{thma*}{Theorem A}
\newtheorem*{thmb*}{Theorem B}
\newtheorem*{thmc*}{Theorem C}

\theoremstyle{remark}
\newtheorem{exm}[equation]{Example}

\newtheorem{rmk}[equation]{Remark}

\newenvironment{enumroman}
{\begin{enumerate}\renewcommand{\labelenumi}{\textnormal{(\roman{enumi})}}}
{\end{enumerate}}

\DeclareMathOperator{\Aut}{Aut}

\DeclareMathOperator{\Gal}{Gal}
\DeclareMathOperator{\Hol}{Hol}

\DeclareMathOperator{\Mon}{Mon}
\DeclareMathOperator{\N}{N}
\DeclareMathOperator{\opdiv}{div}

\DeclareMathOperator{\PGL}{PGL}
\DeclareMathOperator{\PSL}{PSL}
\DeclareMathOperator{\PSU}{PSU}

\DeclareMathOperator{\SL}{SL}
\DeclareMathOperator{\Sym}{Sym}
\DeclareMathOperator{\GL}{GL}

\newcommand{\C}{\mathbb C}
\newcommand{\F}{\mathbb F}
\newcommand{\PP}{\mathbb P}
\newcommand{\Q}{\mathbb Q}
\newcommand{\R}{\mathbb R}
\newcommand{\Z}{\mathbb Z}
\newcommand{\Qbar}{\overline{\mathbb Q}}

\newcommand{\Belyi}{Bely\u{\i}}

\newcommand{\frakp}{\mathfrak{p}}

\newcommand{\calA}{\mathcal{A}}
\newcommand{\calD}{\mathcal{D}}

\newcommand{\calH}{\mathcal{H}}
\newcommand{\calM}{\mathcal{M}}

\newcommand{\psmod}[1]{~(\textup{\text{mod}}~{#1})}

\newcommand{\la}{\langle}
\newcommand{\ra}{\rangle}

\newcommand{\sbnum}{{}_{\textup{num}}}

\newcommand{\defi}{\textsf}

\begin{document}

\title{On computing Belyi maps}

\author{Jeroen Sijsling}
\address{Mathematics Institute, Zeeman Building, University of Warwick, Coventry
  CV4 7AL, UK}
\email{sijsling@gmail.com}
\author{John Voight}
\address{Department of Mathematics and Statistics, University of Vermont, 16
  Colchester Ave, Burlington, VT 05401, USA; Department of Mathematics,
  Dartmouth College, 6188 Kemeny Hall, Hanover, NH 03755, USA}
\email{jvoight@gmail.com}
\date{\today}

\begin{abstract}
  We survey methods to compute three-point branched covers of the projective
  line, also known as \Belyi\ maps. These methods include a direct approach,
  involving the solution of a system of polynomial equations, as well as complex
  analytic methods, modular forms methods, and $p$-adic methods.  Along the way,
  we pose several questions and provide numerous examples.\\

  Nous donnons un aper\c{c}u des m\'ethodes actuelles pour le calcul des
  rev\^etements de la droite projective ramifi\'es sur au plus trois points, aussi
  appel\'es les morphismes de \Belyi. Ces m\'ethodes comprennent une approche
  directe, qui revient \`a la solution d'un syst\`eme d'\'equations polynomiales,
  ainsi que des m\'ethodes analytiques complexes, m\'ethodes de formes modulaires, et m\'ethodes $p$-adiques. En chemin,
  nous posons quelques questions et donnons de nombreux exemples.
\end{abstract}

\maketitle
\setcounter{tocdepth}{1}

\tableofcontents

\section*{Introduction}

Every compact Riemann surface $X$ is an algebraic curve over $\C$, and every
meromorphic function on $X$ is an algebraic function.  This remarkable fact,
generalized in the GAGA principle, links the analytic with the algebraic in a
fundamental way.  A natural problem is then to link this further with
arithmetic; to characterize those Riemann surfaces that can be defined by
equations over $\Qbar$ and to study the action of the absolute Galois group
$\Gal(\Qbar/\Q)$ on these algebraic curves.  To this end, \Belyi\
\cite{Belyi,Belyi2} proved that a Riemann surface $X$ over $\C$ can be defined
over $\Qbar$ if and only if $X$ admits a \defi{\Belyi\ map}, a map $f:X \to
\PP^1_\C$ that is unramified away from $\{0,1,\infty\}$.  Grothendieck, in his
\emph{Esquisse d'un Programme} \cite{Grothendieck}, called this result ``deep
and disconcerting.''  

Part of Grothendieck's fascination with \Belyi 's theorem was a consequence of
the simple combinatorial and topological characterization that follows from it.
Given a \Belyi\ map $f:X \to \PP^1_\C$, the preimage $f^{-1}([0,1])$ of the real
interval $[0,1]$ can be given the structure of a \defi{dessin} (or \defi{dessin
d'enfant}, ``child's drawing''): a connected graph with bicolored vertices (so
the two vertices of an edge are colored differently) equipped with a cyclic
ordering of the edges around each vertex.  Conversely, a
dessin determines the corresponding \Belyi\ map uniquely up to isomorphism over
$\C$ or $\Qbar$.  The idea that one can understand the complicated group
$\Gal(\Qbar/\Q)$ by looking at children's pictures casts an alluring spell
indeed.  As a consequence, hundreds of papers have been written on the subject,
several books have appeared, and the topic remains an active area of research
with many strands.

In a number of these papers, computation of particular examples plays a key role
in understanding phemonena surrounding \Belyi\ maps; arguably, part of the
richness of the subject lies in the beauty in these examples.  Shabat and
Voevodsky \cite[0.1.1, 0.3]{SV} say on this point: 
\begin{quote}
  Here we have no general theory and only give a number of examples.  The
  completeness of our results decrease rapidly with growing genus; we are able
  to give some complete lists (of non-trivial experimental material) for genus
  $0$, but for genera exceeding $3$ we are able to give only some general
  remarks. [...] The main reasons to publish our results in the present state is
  our eagerness to invite our colleagues into the world of the divine beauty and
  simplicity we have been living in since we have been guided by the Esquisse.
\end{quote}

In spite of this important role, no survey of computational methods for \Belyi\
maps has yet appeared, and in our own calculations we found many techniques,
shortcuts, and some tricks that others had also (re)discovered.  In this
article, we collect these results in one place in the hope that it will be
useful to others working in one of the many subjects that touch the theory of
\Belyi\ maps. We also give many examples; to our knowledge, the larger examples
are new, unless mentioned otherwise.  We assume that the reader has some
familiarity with algebraic curves and with computation, but not necessarily with
the theory of \Belyi\ maps or dessins; at the same time, we hope that this paper
will also be a useful and comprehensive reference, so we will also make some
remarks for the experts.

We take as input to our methods the simple group theoretic description of
\Belyi\ maps: there is a bijection between \defi{permutation triples}
\begin{center}
  $\sigma = (\sigma_0, \sigma_1, \sigma_{\infty}) \in S_d^3$ that
  satisfy $\sigma_0 \sigma_1 \sigma_{\infty} = 1$
\end{center}
up to simultaneous conjugation in the symmetric group $S_d$, and
\begin{center}
  \Belyi\ maps $f:X \to \PP^1$ of degree $d$
\end{center}
up to isomorphism over $\Qbar$.  In this bijection, the curve $X$ can be
disconnected, such as the trivial cover of degree $d > 1$; the cover $X$ is
connected if and only if (the dessin is connected if and only if) the
corresponding permutation triple $\sigma$ generates a transitive subgroup of
$S_d$, in which case we call $\sigma$ \defi{transitive}.  If $\sigma$
corresponds to $f$ in this bijection, we say that $f$ has \defi{monodromy
representation $\sigma$}.

Given the description of a \Belyi\ map $f$ in the compressed form of a
permutation triple, it has proven difficult in general to determine explicitly
an algebraic model for $f$ and the curve $X$.  As a result, many authors have
written on this subject of explicit computation of \Belyi\ maps, usually subject
to certain constraints or within a certain class of examples.  That this is a
difficult problem is a common refrain, and the following quote by Magot and
Zvonkin \cite[\S 1]{MagotZvonkin} is typical:
\begin{quote}
  An explicit computation of a Belyi function corresponding to a given map is
  reduced to a solution of a system of algebraic equations.  It may turn out to
  be extremely difficult.  To give an idea of the level of difficulty, we
  mention that our attempts to compute Belyi functions for some maps with only
  six edges took us several months, and the result was achieved only after using
  some advanced Gr\"obner bases software and numerous consultations given by its
  author J.C. Faug\`ere. 
\end{quote} 

The paper is organized as follows.  In Section 1, we collect the basic
background (including a discussion of fields of definition), and mention some
applications and generalizations.  In Section 2, we discuss a direct method
using Gr\"obner methods, augmented by the Atkin--Swinnerton-Dyer trick.  We then
turn to other, more practical methods.  We begin in Section 3 with complex
analytic methods; in Section 4, we consider methods using modular forms; in
Section 5, we consider $p$-adic methods.  In Section 6, we briefly discuss
alternative methods for Galois \Belyi\ maps.  In Section 7, we discuss the
delicate subjects of field of moduli and field of definition with an eye to its
implications for computation.  In Section 8, we treat simplification and
verification of \Belyi\ maps, and finally in Section 9 we conclude by
considering some further topics and generalizations.  Along the way, we give
explicit examples and pose several questions.  

The authors would like to thank Noam Elkies, Ariyan Javanpeykar, Curtis
McMullen, John McKay, David Roberts, Steffen Rohde, Sam Schiavone, Matthias
Sch\"utt, Marco Streng, Bernd Sturmfels, Mark Watkins and Bruce Westbury for their comments on
this work, as well as the referee for his or her many suggestions.  The first
author was supported by Marie Curie grant IEF-GA-2011-299887, and the second
author was supported by an NSF CAREER Award (DMS-1151047).

\section{Background and applications}\label{sec:backg}

The subject of explicit characterization and computation of ramified covers of
Riemann surfaces is almost as old as Riemann himself. Klein \cite{Klein} and
Fricke--Klein \cite{FrickeKlein} calculated some explicit \Belyi\ maps, most
notably the icosahedral Galois \Belyi\ map $\PP^1 \rightarrow \PP^1$ of degree
$60$ \cite[I, 2, \S 13--14]{Klein}. These appeared when constructing what we
would today call modular functions associated with the triangle groups
$\Delta(2,3,5)$ and $\Delta(2,4,5)$ (see Section \ref{sec:mod}). This in turn
allowed them to find a solution to the quintic equation by using analytic
functions. Around the same time, Hurwitz \cite{Hurwitz} was the first to
consider ramified covers in some generality: besides considering covers of small
degree, he was the first to give the classical combinatorical description of
covers of the projective line minus a finite number of points, which would later
result in Hurwitz spaces being named after him.

Continuing up to the modern day, the existing literature on \Belyi\ maps with an
explicit flavor is extremely rich: surveys include Birch \cite{Birch},
Jones--Singerman \cite{JonesSingerman,JonesSingerman2}, Schneps \cite{Schneps},
and Wolfart \cite{WolfartDessins}; textbooks include the seminal conference
proceedings \cite{Schnepsbook}, work of Malle--Matzat \cite{MalleMatzat}, Serre
\cite{Serre}, and V\"olklein \cite{Volklein}, mainly with an eye toward
applications to inverse Galois theory, the tome on graphs on surfaces by
Lando--Zvonkin \cite{LandoZvonkin}, and the book by Girondo--Gonzalez-Diaz
\cite{GDG}, which interweaves the subject with an introduction to Riemann
surfaces.

We begin this section by reviewing basic definitions; we conclude by mentioning
applications and generalizations as motivation for further study.  (We postpone
some subtle issues concerning fields of moduli and fields of definition until
Section \ref{sec:fomfod}.)

\subsection*{Definitions, and equivalent categories}

Let $K$ be a field with algebraic closure $\overline{K}$.  An \defi{(algebraic)
curve} $X$ over $K$ is a smooth proper separated scheme of finite type over $K$
that is pure of dimension $1$.  

We now define precisely the main category of this paper whose objects we wish to
study. A \defi{\Belyi\ map over $K$} is a morphism $f:X \to \PP^1$ of curves
over $K$ that is unramified outside $\{0,1,\infty\}$. Given two \Belyi\ maps
$f_1,f_2:X_1,X_2 \to \PP^1$, a \defi{morphism} of \Belyi\ maps from $f_1$ to
$f_2$ is a morphism $g : X_1 \to X_2$ such that $f_1 = f_2 g$. We thereby obtain
a category of \Belyi\ maps over $K$. 

A curve $X$ that admits a \Belyi\ map is called a \defi{\Belyi\ curve}.  \Belyi\
\cite{Belyi,Belyi2} proved that a curve $X$ over $\C$ can be defined over
$\Qbar$ if and only if $X$ is a \Belyi\ curve.  Consequently, in what follows,
we may pass freely between \Belyi\ maps over $\Qbar$ and over $\C$: we will
simply refer to both categories as the category of \defi{\Belyi\ maps}.  The
absolute Galois group $\Gal (\Qbar /\Q)$ acts naturally on (the objects and
morphisms in) the category of \Belyi\ maps (over $\Qbar$); this action is
faithful, as one can see by considering the $j$-invariant of elliptic curves.
We denote the action by a superscript on the right, so the conjugate of a curve
$X$ over $\Qbar$ by an automorphism $\tau \in \Gal (\Qbar/\Q)$ is denoted by
$X^{\tau}$, and that of a \Belyi\ map $f$ by $f^{\tau}$.

Let $f:X \to \PP^1$ be a \Belyi\ map of degree $d$.  The ramification of $f$
above $\{0,1,\infty\}$ is recorded in its \defi{ramification type}, the triple
consisting of the set of ramification multiplicities above $0,1,\infty$,
respectively. Such a ramification type is therefore given by a triple of
partitions of $d$, or alternatively by a triple of conjugacy classes in the
symmetric group $S_d$. 

Part of the beauty of subject of \Belyi\ maps is the ability to pass seamlessly
between combinatorics, group theory, algebraic geometry, topology, and complex
analysis: indeed, one can define categories in these domains that are all
equivalent.  In the remainder of this subsection, we make these
categories and equivalences precise; the main result is Proposition
\ref{prop:cateq1} below.

To begin, we record the ramification data, or more precisely the monodromy.  A \defi{permutation triple} of
\defi{degree} $d$ is a triple $\sigma=(\sigma_0 , \sigma_1 ,
\sigma_{\infty}) \in S_d^3$ such that $\sigma_0 \sigma_1 \sigma_{\infty} = 1$.
Let $\sigma'=(\sigma'_0 , \sigma'_1 , \sigma'_{\infty})$ be another such triple
of degree $d'$. Then a \defi{morphism} of permutation triples from $\sigma$ to
$\sigma'$ is a map $t : \left\{ 1 , \dots , d \right\} \to \left\{ 1 , \dots ,
d' \right\}$ such that  $t(\sigma_0(x))=\sigma_0'(t(x))$ for all $x \in S$ and
the same for $\sigma_1,\sigma_\infty$.  In particular, two permutation triples
$\sigma,\sigma'$ are isomorphic, and we write $\sigma \sim \sigma'$ and say they
are \defi{simultaneously conjugate}, if and only if they have the same degree
$d=d'$ and there exists a $\tau \in S_d$ such that 
\begin{align*}
  \sigma^\tau = \tau^{-1}(\sigma_0,\sigma_1,\sigma_{\infty})\tau
  = (\tau^{-1}\sigma_0\tau,\tau^{-1}\sigma_1\tau,\tau^{-1}\sigma_{\infty}\tau)
  = (\sigma'_0,\sigma'_1,\sigma'_{\infty}).
\end{align*}

It is a consequence of the Riemann existence theorem that the category of
\Belyi\ maps is equivalent to the category of permutation triples. More
precisely, let
\begin{equation} \label{eqn:F2}
  F_2 = \langle x,y,z \mid xyz = 1 \rangle
\end{equation}
be the free group on two generators. Given a group $G$, a \defi{finite
$G$-set} is a homomorphism $\alpha:G \to \Sym(S)$ on a finite set $S$, and a
\defi{morphism} between finite $G$-sets from $\alpha$ to $\alpha'$ is a map of
sets $t : S \to S'$ such that $\alpha'(g)(t(x))=t(\alpha(g)(x))$ for all $g \in
G$ and $x \in S$. We see that giving a permutation triple is the same as giving
a finite $F_2$-set, by mapping $x,y,z \in F_2$ to $\sigma_0 , \sigma_1 ,
\sigma_{\infty}$, and that two permutation triples are isomorphic if and only if
the corresponding $F_2$-sets are isomorphic.

Returning to covers and topological considerations, we have an isomorphism 
\begin{align*}
  F_2 \cong \pi_1(\PP^1 \setminus \{0,1,\infty\}) ;
\end{align*}
the generators $x,y,z$ chosen above can be taken to be simple counterclockwise
loops around $0,1,\infty$.  We abbreviate $\PP^1_*=\PP^1 \setminus
\{0,1,\infty\}$.  The category of finite topological covers of $\PP^1_*$ is
equivalent to the category of finite $\pi_1( \PP^1_*)$-sets; to a cover, we
associate one of its fibers, provided with the structure of $\pi_1(
\PP^1_*)$-set defined by path lifting.  Therefore, a \Belyi\ map gives rise to a
cover of $\PP^1_*$ by restriction, and conversely a finite topological cover of
$\PP^1_*$ can be given the structure of Riemann surface by lifting the complex
analytic structure and thereby yields a map from an algebraic curve to $\PP^1$
unramified away from $\{0,1,\infty\}$.

Let $f$ be a \Belyi\ map, corresponding to a permutation triple $\sigma$.  The
corresponding $F_2$-set $\rho : F_2 \rightarrow S_d$ is called the
\defi{monodromy representation} of $f$, and its image is called the
\defi{monodromy group} of $f$. The monodromy group, as a subgroup of $S_d$, is
well-defined up to conjugacy and in particular up to isomorphism, and we denote
it by $G=\Mon (f)$. By the correspondences above, the automorphism group of a
\Belyi\ map is the centralizer of its monodromy group (as a subgroup of $S_d$).

We consider a final category, introduced by Grothendieck \cite{Grothendieck}.  A
\defi{dessin} $D$ is a triple $(\Gamma,C,O)$ where:
\begin{itemize}
  \item[(D1)] $\Gamma$ is a finite graph with vertex set $V$, edge set $E$, and
    vertex map $v: E \to V \times V$;
  \item[(D2)] $C : V \to \left\{ 0,1 \right\}$ is a bicoloring of the vertices
    such that the two vertices of an edge are colored differently, i.e., $C( v
    (e)) = \left\{ 0,1 \right\}$ (and not a proper subset) for all edges $e \in
    E$; and
  \item[(D3)] $O$ is a cyclic orientation of the edges around every vertex.
\end{itemize}
Due to the presence of the bicoloring $C$, the cyclic orientation in (D3) is
specified by two permutations $O_0 , O_1 \in \Sym (E)$ specifying the orderings
around the edges marked with $0$ and $1$ respectively. Note that once the
bicoloring $C$ is given, the possible orientations $O = (O_0 , O_1)$ can be
chosen to be any pair of permutations with the property that two edges $e,e'$
are in the same orbit under $O_0$ (resp.\ $O_1$) if and only if the
corresponding vertices marked $0$ (resp.\ $1$) coincide.  A \defi{morphism} of
dessins is a morphism of graphs $\varphi:\Gamma \to \Gamma'$ such that $\varphi$
takes the bicoloring $C$ to $C'$ (i.e., $C'(\varphi(v))=C(v)$) and similarly the
cyclic orientation $O$ to $O'$.

The category of dessins is also equivalent to that of \Belyi\ maps. Indeed,
associated to a \Belyi\ map $f$ is the graph given by $f^{-1}([0,1])$, with the
bicoloring on the vertices given by $f$ and with the cyclic ordering induced by
the orientation on the Riemann surface. Conversely, given a dessin we can
algebraize the topological covering induced by sewing on $2$-cells as specified
by the ordering $O$.

Dessins were introduced by Grothendieck \cite{Grothendieck} to study the action
of $\Gal (\Qbar/\Q)$ on \Belyi\ maps through combinatorics. So far, progress has
been slow, but we mention one charming result \cite{Schneps}; the Galois action
is already faithful on the dessins that are \defi{trees} (as graphs). 

We summarize the equivalences obtained in the following proposition and refer to
Lenstra \cite{Lenstra} for further exposition and references.

\begin{prop}\label{prop:cateq1}
  The following categories are equivalent:
  \begin{enumroman}
    \item \Belyi\ maps;
    \item permutation triples;
    \item finite $F_2$-sets; and
    \item dessins.
  \end{enumroman}
\end{prop}

In particular, the equivalence in Proposition \ref{prop:cateq1} yields the key
bijection considered in this paper:
\begin{equation} \label{eq:corresp}
  \begin{gathered} 
    \left\{ \text{permutation triples $\sigma = (\sigma_0, \sigma_1,
      \sigma_{\infty}) \in S_d^3$} \right\} /
      \sim
\\    \updownarrow \text{\small{1:1}} \\
    \left\{ \text{\Belyi\ maps $f:X \to \PP^1$ of degree $d$} \right\} /
      \cong_{\Qbar}
  \end{gathered}
\end{equation}
where the notions of isomorphism are taken in the appropriate categories.
Concretely, under the correspondence \eqref{eq:corresp}, the cycles of the
permutation $\sigma_0$ (resp.\ $\sigma_1,\sigma_{\infty}$) correspond to the
points of $X$ above $0$ (resp.\ $1,\infty$) and the length of the cycle
corresponds to the ramification index of the corresponding point under the morphism $f$.  Note in
particular that because the first set of equivalence classes in
\eqref{eq:corresp} is evidently finite, there are only finitely many
$\Qbar$-isomorphism classes of curves $X$ with a \Belyi\ map of given degree.  


It is often useful, and certainly more intuitive, to consider the subcategories
in Proposition \ref{prop:cateq1} that correspond to \Belyi\ maps $f : X \to
\PP^1$ whose source is connected (and accordingly, we say the map is
\defi{connected}).  A \Belyi\ map is connected if and only if the corresponding
permutation triple $\sigma$ is \defi{transitive}, i.e., the subgroup $\langle
\sigma_0,\sigma_1,\sigma_{\infty} \rangle$ is a transitive group.  Restricting
to transitive permutations gives a further equivalent category of finite index
subgroups of $F_2$: the objects are subgroups $H \leq F_2$ of finite index and
morphisms $H \to H'$ are restrictions of inner automorphisms of $F_2$ that map
$H$ to $H'$.  The category of finite index subgroups of $F_2$ is equivalent to
that of finite transitive $F_2$-sets (to a subgroup $H$ of $F_2$, one associates
the $F_2$-set $F_2 / H$).  Proposition \ref{prop:cateq1} now becomes the
following.

\begin{prop}\label{prop:cateq2}
  The following categories are equivalent:
  \begin{enumerate}
    \item[(i)] connected \Belyi\ maps;
    \item[(ii)] transitive permutation triples;
    \item[(iii)] transitive finite $F_2$-sets;
    \item[(iii$'$)] subgroups of $F_2$ of finite index; and
    \item[(iv)] dessins whose underlying graph is connected.
  \end{enumerate}
\end{prop}

Unless stated otherwise (e.g., Section \ref{sec:fomfod}), in the rest of this
article we will assume without further mention that a \Belyi\ map is
\emph{connected}; this is no loss of generality, since any disconnected \Belyi\
map is the disjoint union of its connected components.  

\subsection*{Geometric properties and invariants}

Let $f:X \to \PP^1$ be a (connected) \Belyi\ map over $\Qbar$. If the cover $f$
is \defi{Galois}, which is to say that the corresponding extension of function
fields $\Qbar (X) / \Qbar (\PP^1)$ is Galois, then we call $f$ a \defi{Galois
\Belyi\ map}. More geometrically, this property boils down to the demand that a
subgroup of $\Aut (X)$ act transitively on the sheets of the cover; and
combinatorially, this is nothing but saying that $\Mon (f) \subseteq S_d$ has
cardinality $\#\Mon(f)=d$.  Indeed, the monodromy group of a \Belyi\ map can
also be characterized as the Galois group of its \defi{Galois closure}, which is
the smallest Galois cover of which it is a quotient.

The genus of $X$ can be calculated by using the Riemann--Hurwitz formula.  If we
define the \defi{excess} $e(\tau)$ of a cycle $\tau \in S_d$ to be its length
minus one, and the excess $e(\sigma)$ of a permutation to be the sum of the
excesses of its constituent disjoint cycles (also known as the \defi{index} of
the permutation, equal to $n$ minus the number of orbits), then the genus of a
\Belyi\ map of degree $n$ with monodromy $\sigma$ is
\begin{equation} \label{eq:genus}
  g = 1-n + \frac{e(\sigma_0) + e(\sigma_1) + e(\sigma_{\infty})}{2}.
\end{equation}
In particular, we see that the genus of  \Belyi\ map is zero if and only if
$e(\sigma_0) + e(\sigma_1) + e(\sigma_{\infty}) = 2n - 2$.

We employ exponential notation to specify both ramification types and conjugacy
classes in $S_d$. So for example, if $d = 10$, then $3^2 2^1 1^2$ denotes both
the conjugacy class of the permutation $(1\ 2\ 3)(4\ 5)(6\ 7\ 8)$ and the
corresponding ramification type; two points of ramification index $3$, one of
index $2$, and two (unramified) of index $1$.

The \defi{passport} of a \Belyi\ map $f:X \to \PP^1$ is the triple $(g,G,C)$
where $g$ is the genus of $X$ and $G \subseteq S_d$ is the monodromy group of
$f$, and $C = (C_0 , C_1, C_{\infty})$ is the triple of conjugacy classes of
$(\sigma_0 , \sigma_1 , \sigma_{\infty})$ in $S_d$, respectively
\cite[Definition 1.1.7]{LandoZvonkin}.  Although the genus of the \Belyi\ map is
determined by the conjugacy classes by equation \eqref{eq:genus}, we still
include it in the passport for clarity and ease.  The \defi{size} of a passport
$(g,G,C)$ is the number of equivalence classes of triples $\sigma =
(\sigma_0,\sigma_1,\sigma_{\infty})$ such that $\langle \sigma \ra = G$ and
$\sigma_i \in C_i$ for $i=0,1,\infty$.

We will occasionally need slightly altered notions of passport. The
\defi{ramification passport} of $f$ is the pair $(g,C)$ with conjugacy classes
in $S_d$. Another version of the passport will be considered in Section
\ref{sec:fomfod}. The passport has the following invariance property
\cite{JonesStreit}.

\begin{thm}\label{thm:galinv}
  The passport and the ramification passport of a \Belyi\ map are invariant
  under the action of $\Gal (\Qbar/\Q)$.
\end{thm}

One can calculate the set of isomorphism classes of permutation triples with
given passport using the following lemma with $G = S_d$.

\begin{lem}\label{lem:doubcos}
  Let $G$ be a group and let $C_0,C_1$ be conjugacy classes in $G$ represented
  by $\tau_0,\tau_1 \in G$. Then the map
\begin{align*}
    C_G (\tau_0) \backslash G / C_G (\tau_1) &\to \{ (\sigma_0,\sigma_1) : 
    \sigma_0 \in C_0 , \sigma_1 \in C_1 \} /\! \sim_G \\
    C_G(\tau_0) g C_G(\tau_1) &\mapsto (\tau_0, g \tau_1 g^{-1}) \nonumber
  \end{align*}
  is a bijection, where $C_G (\tau)$ denotes the centralizer of $\tau$ in $G$ and $\sim_G$ denotes
  simultaneous conjugation in $G$.
\end{lem}

The virtue of this lemma is that double-coset methods in group theory are quite
efficient; by using this bijection and filtering appropriately \cite[Lemma
1.11]{KMSV}, this allows us to enumerate \Belyi\ maps with a given passport
relatively quickly up to moderate degree $d$.  One can also estimate the size of
a passport using character theory; for more on this, see Section
\ref{sec:fomfod}.

\subsection*{Applications}

Having introduced the basic theory, we now mention some applications of the
explicit computation of \Belyi\ maps.  

We began in the introduction with the motivation to uncover the mysterious
nature of the action of $\Gal(\Qbar/\Q)$ on dessins following Grothendieck's
\emph{Esquisse}. Dessins of small degree tend to be determined by their passport
in the sense that the set of dessins with given passport forms a full Galois
orbit.  However, even refined notions of passport do not suffice to distinguish
Galois orbits of dessins of high degree in general: a first example was
\defi{Schneps' flower} \cite[\S IV, Example I]{Schneps}. Some further examples
of distinguishing features of non-full Galois orbits have been found by Wood
\cite{Wood} and Zapponi \cite{Zapponi}, but it remains a challenge to determine
the Galois structure for the set of dessins with given passport. Even statistics
in small degree are not known yet; an important project remains to construct
full libraries of dessins.  The original ``flipbook'' of dessins, due to
B\'etr\'ema--P\'er\'e--Zvonkin \cite{BPZ}, contained only dessins that were
plane trees but was already quite influential, and consequently systematic
tabulation promises to be just as inspiring.

Further applications of the explicit study of \Belyi\ maps have been found in
inverse Galois theory, specifically the regular realization of Galois groups
over small number fields: see the tomes of Matzat \cite{Matzat}, Malle--Matzat
\cite{MalleMatzat}, and Jensen--Ledet--Yui \cite{JensenLedetYui}.  Upon
specialization, one obtains Galois number fields with small ramification set:
Roberts \cite{Roberts1,Roberts2,Robertsx}, Malle--Roberts \cite{MalleRoberts},
and Jones--Roberts \cite{JonesRoberts} have used the specialization of
three-point covers to exhibit number fields with small ramification set or root
discriminant. The covering curves obtained are often interesting in their own
right, spurring further investigation in the study of low genus curves (e.g.,
the decomposition of their Jacobian \cite{Paulhus}).  Finally, a \Belyi\ map
$f:\PP^1 \to \PP^1$, after precomposing so that $\{0,1,\infty\} \subseteq
f^{-1}(\{0,1,\infty\})$, is an example of a rigid \defi{post-critically finite
map}, a map of the sphere all of whose critical points have finite orbits.
(Zvonkin calls these maps \defi{dynamical \Belyi\ functions} \cite[\S
6]{Zvonkin}.)  These maps are objects of central study in complex dynamics
\cite{Bartholdi,Pilgrim}: one may study the associated Fatou and Julia sets.  

\Belyi\ maps also figure in the study of \defi{Hall polynomials}, (also called
\defi{Davenport-Stothers triples}) which are those coprime solutions $X(t),Y(t),Z(t) \in
\C[t]$ of the equations in polynomials 
\begin{align*}
  X(t)^3 - Y(t)^2 = Z(t)
\end{align*}
with $\deg(X(t)) = 2 m$, $\deg (Y(t)) = 3 m$ and $\deg (Z(t)) = m + 1$. These
solutions are extremal in the degree of $Z$ and are analogues of \defi{Hall
triples}, i.e.\ integers $x , y \in \Z$ for which $|x^3 - y^2| =O(\sqrt{|x|})$.
Hall polynomials have been studied by Watkins \cite{Watkins} and by
Beukers--Stewart \cite{BeukersStewart}; Montanus \cite{Montanus} uses the
link with dessins ($X^3 (t) / Y^2 (t)$ is a \Belyi\ map) to find a formula for
the number of Hall polynomials of given degree.  Hall polynomials also lead to
some good families of classical Hall triples \cite{ElkiesWatkins}, as 
the following example illustrates.

\begin{exm}
  Taking $m = 5$ above, one obtains the following Hall polynomials due to Birch:
  \begin{align*}
    X(t) & = \frac{1}{9} (t^{10} +6 t^7 +15 t^4 + 12 t)  , \\
    Y(t) & =
    \frac{1}{54} (2 t^{15} + 18 t^{12} + 72 t^9 + 144 t^6 + 135 t^3 + 27)  , \\
    Z(t) & = -\frac{1}{108} (3 t^6 + 14 t^3 + 27).
  \end{align*}

  Choosing $t \equiv 3$ mod $6$, we get some decent Hall triples, notably
  \begin{align*}
    |384242766^3 - 7531969451458^2| &= 14668 \\
    |390620082^3 - 7720258643465^2| &= 14857 
  \end{align*}
  for $t=\pm 9$; remarkably, in both cases the constant factor $|x^3 -
  y^2| / \sqrt{|x|}$ is approximately equal to the tiny number $3/4$.
\end{exm}

\Belyi\ maps also give rise to interesting algebraic surfaces.  The \Belyi\ maps
of genus $0$ and degree $12$ (resp.\ $24$) with ramification indices above $0,1$
all equal to $3,2$ correspond to elliptic fibrations of rational (resp.\ K3)
surfaces with only $4$ (resp.\ $6$) singular fibers; given such a fibration, the
associated \Belyi\ map is given by taking its $j$-invariant. By work of
Beauville \cite{Beauville} (resp.\ Miranda and Persson \cite{MP}), there are $6$
(resp.\ $112$) possible fiber types for these families. This result comes down
to calculating the number of \Belyi\ maps of given degrees with specified conjugacy
classes with cycle type $(3, \dots , 3)$ and $(2, \dots , 2)$ for $\sigma_0$ and
$\sigma_1$.

Especially in the degree $24$ case, the explicit calculation of these \Belyi\ maps is
quite a challenge. By developing clever methods specific to this case, this
calculation was accomplished by Beukers--Montanus \cite{BeukersMontanus}.  They
find $191$ \Belyi\ maps, exceeding the $112$ ramification types determined by Miranda
and Persson: this is an instance of the phenomenon mentioned above, that the
passport may contain more than one \Belyi\ map, so that to a given ramification
triple there may correspond multiple isomorphism classes of \Belyi\ maps.

One can also specialize \Belyi\ maps to obtain $abc$ triples: this connection is
discussed by Elkies \cite{ElkiesABC} and Frankenhuysen \cite{Frankenhuysen} to
show that the $abc$ conjecture implies the theorem of Faltings, and it is also
considered by van Hoeij--Vidunas \cite[Appendix D]{vHV2}.

Modular curves and certain Shimura curves possess a natural \Belyi\ map.
Indeed, Elkies has computed equations for Shimura curves in many cases using
only the extant structure of a \Belyi\ map \cite{ElkiesSCC,Elkies237}. Another
such computation was made by Hallouin in \cite{HallouinComputation}, where a
more elaborate argument using Hurwitz spaces of four-point covers is used.
Explicit equations are useful in many contexts, ranging from the resolution of
Diophantine equations to cryptography \cite{Schoof}.  Reducing these equations
modulo a prime also yields towers of modular curves that are useful in coding
theory.  Over finite fields of square cardinality $q$, work of Ihara
\cite{Ihara} and Tsfasman--Vl\u{a}du\c{t}--Zink \cite{TVZ} shows that modular
curves have enough \defi{supersingular points} that their total number of
rational points is asymptotic with $(\sqrt{q}-1)g$ as their genus grows; this is
asymptotically optimal by work of Drinfeld--Vl\u{a}du\c{t}
\cite{DrinfeldVladut}. By a construction due to Goppa \cite{Goppa}, one obtains
the asympotically best linear error-correcting codes known over square fields.
But to construct and use these codes we need explicit equations for the curves
involved.  A few of these modular towers were constructed by Elkies
\cite{ElkiesModularTowers}.  There are extensions to other arithmetic triangle
towers, using the theory of Shimura curves, which give other results over prime
power fields of larger exponent \cite{DucetThesis}.  For the cocompact triangle
quotients, the modular covers involved are \Belyi\ maps, and in fact many
congruence towers are unramified (and cyclic) after a certain point, which makes
them particularly pleasant to work with.

There are also applications of explicit \Belyi\ maps to algebraic solutions of
differential equations \cite{LarussonSadykov}: as we will see in Section
\ref{sec:mod}, subgroups of finite index of triangle groups correspond to
certain \Belyi\ maps, and the uniformizing differential equations for these
groups (resp.\ their solutions) can be obtained by pulling back suitable
hypergeometric differential equations (resp.\ hypergeometric functions).  Kitaev
\cite{Kitaev} and Vidunas--Kitaev \cite{VK} consider branched covers at $4$
points with all ramification but one occuring above three points (``almost
\Belyi\ coverings'') and apply this to algebraic Painlev\'e VI functions.
Vidunas--Filipuk \cite{VidFil} classify coverings yielding transformations
relating the classical hypergeometric equation to the Heun differential
equation; these were computed by van Hoeij--Vidunas \cite{vHV2,vHV}.  

There are applications to areas farther from number theory. Eyral--Oka
\cite{EyralOka} explicitly use dessins (and their generalizations to covers of
$\PP^1$ branched over more than $3$ points) in their classification of the
fundamental groups of the complement in the projective plane of certain
join-type sextic curves of the form $a \prod_i (X - \alpha_i Z) = b \prod_j (X -
\beta_j Z)$. Boston \cite{Boston} showed how three-point branched covers arise
in control theory, specifically with regards to a certain controller design
equation.  Finally, dessins appear in physics in the context of brane tilings
\cite{Hananyetal} and there is a moonshine correspondence between genus $0$
congruence subgroups of $\SL_2(\Z)$, associated with some special dessins,
and certain representations of sporadic
groups, with connections to gauge theory \cite{HeMcKay1,HeMcKay2,HeMcKayRead}.  

\section{Gr\"obner techniques}\label{sec:groeb}

We now begin our description of techniques for computing \Belyi\ maps.  We start
with the one that is most straightforward and easy to implement, involving the
solutions to an explicit set of equations over $\Q$. For \Belyi\ maps of small
degree, this method works quite well, and considerable technical effort has made
it work in moderate degree. However, for more complicated \Belyi\ maps, it will be
necessary to seek out other methods, which will be described in the sections that follow.

\subsection*{Direct calculation} \label{sec:directcalc}

The direct method has been used since the first \Belyi\ maps were written down,
and in small examples (typically with genus $0$), this technique works well
enough.  A large number of authors describe this approach, with some variations
relevant to the particular case of interest.  Shabat--Voevodsky \cite{SV} and
Atkin--Swinnerton-Dyer \cite{ASD} were among the first.  Birch \cite[Section
4.1]{Birch} computes a table for covers of small degree and genus.  Schneps
\cite[III]{Schneps} discusses the case of clean dessins of genus $0$ and trees.
Malle \cite{MalleTPR} computed a field of definition for many \Belyi\ maps of
small degree and genus $0$ using Gr\"obner methods, with an eye toward
understanding the field of definition of regular realizations of Galois groups
and a remark that such fields of definition also give rise to number fields
ramified over only a few very small primes.  Malle--Matzat \cite[\S
I.9]{MalleMatzat} use a direct method to compute several \Belyi\ maps in the
context of the inverse Galois problem, as an application of rigidity.
Granboulan studied the use of Gr\"obner bases for genus $0$ \Belyi\ maps in
detail in his Ph.D. thesis \cite{GranboulanThesis}. Elkies \cite{ElkiesSCC} used
this technique to compute equations for Shimura curves.  Other authors who have
used this method are Hoshino \cite{Hoshino} (and Hoshino--Nakamura
\cite{HoshinoNakamura}), who computed the non-normal inclusions of triangle
groups (related to the \Belyi{}-extending maps of Wood \cite{Wood}).  Couveignes
\cite[\S 2]{CouveignesTools} also gives a few introductory examples.

We explain how the method works by example in the simplest nontrivial case.

\begin{exm} \label{exm:firstex}
  Take the transitive permutation triple $\sigma=((1\ 2), (2\ 3), (1\ 3\ 2))$
  from $S_3$, with passport $(0,S_3,(2^1 1^1,2^1 1^1,3^1))$.  Since these
  permutations generate the full symmetric group $S_3$, the monodromy group of
  this \Belyi\ map is $S_3$.  The Riemann--Hurwitz formula \eqref{eq:genus} gives the
  genus as
  \begin{align*}
    g = 1 - 3 + \frac{1}{2}(1 + 1 + 2) = 0.
  \end{align*}
  So the map $f : X\cong\PP^1 \to \PP^1$ is given by a rational function $f(t)
  \in K(t)$ where $K \subset \Qbar$ is a number field.  There are two points
  above $0$, of multiplicities $2,1$, the same holds for $1$, and there is a
  single point above $\infty$ with multiplicity $3$.  The point above $\infty$
  is a triple pole of $f(t)$; since it is unique, it is fixed by
  $\Gal(\overline{K}/K)$; therefore we take this point also to be $\infty$,
  which we are free to do up to automorphisms of $\PP_K^1$, and hence $f(t) \in
  K[t]$.  Similarly, the ramified points above $0$ and $1$ are also unique, so
  we may take them to be $0$ and $1$, respectively.  Therefore, we have
  \begin{align*}
    f(t) = ct^2(t+a) 
  \end{align*}
  for some $a,c \in K \setminus \{0\}$ and
  \begin{align*}
    f(t)-1 = c(t-1)^2(t+b) 
  \end{align*}
  for some $b \in K \setminus \{0,-1\}$.  Combining these equations, we get
  \begin{align*}
    c t^2(t+a) - 1 = c(t^3+at^2) - 1 = c(t-1)^2(t+b) = c(t^3 + (b-2)t^2 +
    (1-2b)t + b) 
  \end{align*}
  and so by comparing coefficients we obtain $b=1/2$, $c=-2$, and $a=-3/2$.  In
  particular, we see that the map is defined over $K=\Q$ and is unique up to
  $\Aut(\PP_\Q^1) \cong \PGL_2(\Q)$.  Thus
  \begin{align*}
    f(t) = -t^2(2t-3) = -2t^3+3t^2, \quad f(t)-1 = -(t-1)^2(2t+1). 
  \end{align*}
  If we relax the requirement that the ramification set be $\{0,1,\infty\}$ and
  instead allow $\{0,r,\infty\}$ for some $r \neq 0,\infty$, then the form of
  $f$ can be made more pleasing. For example, by taking $f(t)=t^2(t+3)$ and
  $r=4$ we obtain $f(t)-4=(t-1)^2(t+2)$.  
\end{exm}

It is hopefully clear from this example (see Schneps \cite[Definition
8]{Schneps}) how to set up the corresponding system of equations for a \Belyi\
map on a curve of genus $g=0$: with variable coefficients, we equate the two
factorizations of a rational map with factorization specified by the cycle types
in the permutations triple $\sigma$.  We illustrate this further in the
following example; for a large list of examples of this kind, see Lando--Zvonkin
\cite[Example 2.3.1]{LandoZvonkin}.

\begin{exm} \label{exm:largegrob}
  To get a small taste of how complicated the equations defining a passport can get, consider the case
  $G=\PGL_2(\F_7)$ with permutation triple $\sigma=(\sigma_0,\sigma_1,\sigma_{\infty})$ given by
  \begin{align*}
    \sigma_0 = \begin{pmatrix} -1 & 0 \\ 0 & 1 \end{pmatrix} , \quad 
    \sigma_1 = \begin{pmatrix} -1 & 1 \\ -1 & 0 \end{pmatrix} , \quad 
    \sigma_{\infty} = \begin{pmatrix} 0 & 1 \\ 1 & 1 \end{pmatrix}.  
  \end{align*}
  The permutation representation of $G$ acting on the set of $8$ elements
  $\PP^1(\F_7)$ is given by the elements
  \begin{align*}
    (1\ 6)(2\ 5)(3\ 4), \quad (0\ \infty\ 1)(2\ 4\ 6), \quad (0\ 1\ 4\ 3\ 2\ 5\
    6\ \infty). 
  \end{align*}
  The corresponding degree $8$ \Belyi\ map $f:X \to \PP^1$ has passport 
  \[ (0,
  \PGL_2 (\F_7), ( 2^3 1^2, 3^2 1^2 , 8^1)). \]
  After putting the totally ramified
  point at $\infty$, the map $f$ is given by a polynomial $f(t) \in \Qbar[t]$
  such that
  \begin{equation}\label{eqn:ff1}
    f(t) = c a(t)^2 b(t) \quad \text{ and } \quad f(t) - 1 = c d(t)^3 e(t)
  \end{equation}
  where $c \in \Qbar^\times$ and $a(t),b(t),d(t),e(t) \in \Qbar[t]$ are monic
  squarefree polynomials with $\deg a(t)=3$ and $\deg b(t)=\deg d(t)=\deg e(t) =
  2$.  We write $a(t)=t^3+a_2t^2+a_1t+a_0$, etc.  

  Equating coefficients in (\ref{eqn:ff1}) we obtain the following system of $8$
  vanishing polynomials in $10$ variables:
  \begin{align*}
    &    a_0^2 b_0 c - c d_0^3 e_0, \\
    &    2 a_1 a_0 b_0 c + a_0^2 b_1 c - 3 c d_1 d_0^2 e_0 - c d_0^3 e_1, \\
    &    2 a_2 a_0 b_0 c + a_1^2 b_0 c + 2 a_1 a_0 b_1 c + a_0^2 c - 3 c d_1^2
         d_0 e_0 - 3 c d_1 d_0^2 e_1 - c d_0^3 - 3 c d_0^2 e_0, \\
    &    2 a_2 a_1 b_0 c + 2 a_2 a_0 b_1 c + a_1^2 b_1 c + 2 a_1 a_0 c + 2 a_0
         b_0 c - c d_1^3 e_0 - 3 c d_1^2 d_0 e_1 - 3 c d_1 d_0^2 \\
    &    \qquad
         - 6 c d_1 d_0 e_0 - 3 c d_0^2 e_1, \\
    &    a_2^2 b_0 c + 2 a_2 a_1 b_1 c + 2 a_2 a_0 c + a_1^2 c + 2 a_1 b_0 c + 2
         a_0 b_1 c -  c d_1^3 e_1 - 3 c d_1^2 d_0 - 3 c d_1^2 e_0 \\
    &    \qquad      - 6 c d_1 d_0 e_1 - 3 c d_0^2 -  3 c d_0 e_0, \\
    &    a_2^2 b_1 c + 2 a_2 a_1 c + 2 a_2 b_0 c + 2 a_1 b_1 c + 2 a_0 c - c
         d_1^3 - 3 c d_1^2 e_1 - 6 c d_1 d_0 - 3 c d_1 e_0 - 3 c d_0 e_1, \\
    &    a_2^2 c + 2 a_2 b_1 c + 2 a_1 c + b_0 c - 3 c d_1^2 - 3 c d_1 e_1 - 3 c
         d_0 - c e_0, \\
    &    2 a_2 c + b_1 c - 3 c d_1 - c e_1.
  \end{align*}

  Using a change of variables $t \leftarrow t-r$ with $r \in \Qbar$ we may
  assume that $b_1=0$, so $b_0 \neq 0$.  Note that if $f(t) \in K[t]$ is defined
  over $K$ then we may take $r \in K$, so we do not unnecessarily increase the
  field of definition of the map. Similarly, if $d_1 \neq 0$, then with $t
  \leftarrow ut$ and $u \in K^\times$ we may assume $d_1=b_0$; similarly if
  $e_1 \neq 0$, then we may take $e_1=b_0$. If $d_1=e_1=0$, then $f(t)=g(t^2)$
  is a polynomial in $t^2$, whence $a_0=0$ so $a_1 \neq 0$, and thus we may take
  $a_1=b_0$.  This gives a total of three cases: (i) $d_1=b_0 \neq 0$, (ii)
  $d_1=0$ and $e_1=b_0 \neq 0$, and (iii) $d_1=e_1=0$ and $a_1=b_0 \neq 0$.  We
  make these substitutions into the equations above, adding $c \neq 0$ and
  $b_1=0$ in all cases. Note that the equation $c \neq 0$ can be added
  algebraically by introducing a new variable $c'$ and adding the equation $c c'
  = 1$.

  These equations are complicated enough that they cannot be solved by hand, but
  not so complicated that they cannot be solved by a Gr\"obner basis.  There are
  many good references for the theory of Gr\"obner bases
  \cite{AdamsLoustaunau,CLO,CLO2,GreuelPfister,KreuzerRobbiano}.

  In the degenerate cases (ii) and (iii) we obtain the unit ideal, which does
  not yield any solutions. In the first case, we find two conjugate solutions
  defined over $\Q(\sqrt{2})$.  After some simplification, the first of the
  solutions becomes
  \begin{align*}
    f(t) = \bigl(2\sqrt{2}t^3-2(2\sqrt{2}+1)t^2+(-4+7\sqrt{2})t+1\bigr)^2
    \bigl(14t^2+6(\sqrt{2}+4)t - 8\sqrt{2}+31\bigr) 
  \end{align*}
  with
  \begin{align*}
    f(t)-432(4\sqrt{2}-5) = \bigl(2t^2-2\sqrt{2}+1 \bigr)^3 \bigl(14t^2
    -8(\sqrt{2}+4)t - 14\sqrt{2}+63 \bigr). 
  \end{align*}
\end{exm}

The direct method does not give an obvious way to discriminate among \Belyi\
maps by their monodromy groups, let alone to match up which Galois conjugate
corresponds to which monodromy triple: all covers with a given ramification type
are solutions to the above system of equations.  

To set up a similar system of equations in larger genus $g \geq 1$, one can for
example write down a general (singular) plane curve of degree equal to $\deg
\varphi$ and ask that have sufficiently many nodal singularities so that it has
geometric genus $g$; the \Belyi\ map can then be taken as one of the
coordinates, and similar techniques apply, though many non-solutions will
still be obtained in this way by cancellation of numerator and denominator. 

\begin{rmk}
  Any explicitly given quasiprojective variety $X$ with a surjective map to the
  moduli space $\calM_g$ of curves of genus $g$ will suffice for this purpose;
  so for those genera $g$ where the moduli space $\calM_g$ has a simpler
  representation (such as $g\leq 3$), one can use this representation instead.
  The authors are not aware of any \Belyi\ map computed in this way with genus
  $g \geq 3$.
\end{rmk}

The direct method can be used to compute the curves $X$ with \Belyi\ maps of
small degree. The curve $\PP^1$ is the only curve with a \Belyi\ map of degree
$2$ (the squaring map), and the only other curve that occurs in degree $3$ is
the genus $1$ curve with $j$-invariant $0$ and equation  $y^2 = x^3 + 1$,  for
which the \Belyi\ map is given by projecting onto the $y$-coordinate.  In degree
$4$, there is the elliptic curve of $j$-invariant 1728 with equation $y^2=x^3-x$
with \Belyi\ map given by $x^2$ and one other given by the the elliptic curve
$y^2 = 4(2x + 9)(x^2 + 2x + 9)$ and regular function $y + x^2 +4 x + 18$.  Both
were described by Birch \cite{Birch}. 

In the direction of tabulating the simplest dessins in this way, all clean
dessins (i.e.\ those for which all ramification indices above $1$ are equal to
$2$) with at most $8$ edges were computed by Adrianov et al.\
\cite{Adrianovetal}.   Magot--Zvonkin \cite{MagotZvonkin} and
Couveignes--Granboulan \cite{CouveignesGranboulan} computed the genus $0$
\Belyi\ maps corresponding to the Archimedean solids, including the Platonic
solids, using symmetry and Gr\"obner bases.  For a very complete discussion of
trees and Shabat polynomials and troves of examples, see Lando--Zvonkin \cite[\S
2.2]{LandoZvonkin}.

In general, we can see that these Gr\"obner basis techniques will present
significant algorithmic challenges.  Even moderately-sized examples, including
all but the first few of genus $1$, do not terminate in a reasonable time.  (In
the worst case, Gr\"obner basis methods have running time that is doubly
exponential in the input size, though this can be reduced to singly exponential
for zero-dimensional ideals; see the surveys of Ayad \cite{Ayad} and Mayr
\cite{Mayr}.) One further differentiation trick, which we introduce in the next
section, allows us to compute in a larger range.  However, even after this
modification, another obstacle remains: the set of solutions can have
positive-dimensional degenerate components. These components correspond to situations where
roots coincide or there is a common factor and are often called \defi{parasitic
solutions} \cite{Kreines03,Kreines08}.  The set of parasitic solutions have been
analyzed in some cases by van Hoeij--Vidunas \cite[\S 2.1]{vHV}, but they remain
a nuisance in general (as can be seen already in Example \ref{exm:largegrob}
above).

\begin{rmk}\label{rmk:naive1}
  Formulated more intrinsically, the naive equations considered in this section
  determine a scheme in the coefficient variables that is a naive version of the
  Hurwitz schemes that will be mentioned in Section \ref{sec:gens}.  Besides
  containing degenerate components, this naive scheme is usually very
  non-reduced. We will revisit this issue in Remark \ref{rmk:naive2}.

  When calculating a \Belyi\ map $f : \PP^1 \to \PP^1$, one usually fixes points
  on the source and the target. As we saw most elaborately when working out
  equation (\ref{eqn:ff1}), this reduces the problem of calculating a \Belyi\
  map in genus $0$ to finding the points on an affine scheme. The families of
  solutions in which numerator and denominator cancel give rise to some of the
  degenerate components mentioned in the previous paragraph.
\end{rmk}  

\subsection*{The ASD differentiation trick}

There is a trick, due to Atkin--Swinnerton-Dyer \cite[\S 2.4]{ASD} that uses the
derivative of $f$ to eliminate a large number of the indeterminates (``the
number of unknowns $c$ can be cut in half at once by observing that $dj/d\zeta$
has factors $F_3^2 F_2$'').  Couveignes \cite{CouveignesTools} implies that this
trick was known to Fricke; it has apparently been rediscovered many times.
Hempel \cite[\S 3]{Hempel} used differentiation by hand to classify subgroups of
$\SL_2(\Z)$ of genus $0$ with small torsion and many cusps.  Couveignes \cite[\S
2,\S 10]{CouveignesCRF} used this to compute examples in genus $0$ of
\defi{clean} dessins.  Schneps \cite[\S III]{Schneps} used this trick to
describe a general approach in genus $0$.  Finally, Vidunas \cite{Vidunas1}
applied the trick to differential equations, and Vidunas--Kitaev \cite{VK}
extended this to covers with $4$ branch points. 

\begin{exm}
  Again we illustrate the method by an example.  Take 
  \begin{align*}
    \sigma=((1\ 2), (2\ 4\ 3), (1\ 2\ 3\ 4))
  \end{align*}
  with passport $(0,S_4,(2^1 1^2, 3^1 1^1 , 4^1))$. Choosing the points $0$ and
  $1$ again to be ramified, this time of degrees $2,3$ above $0,1$ respectively,
  and choosing $\infty$ to be the ramified point above $\infty$, we can write
  \begin{align*}
    f(t) = c t^2(t^2+at+b)
  \end{align*}
  and
  \begin{align*}
    f(t)-1 = c(t-1)^3(t+d). 
  \end{align*}
  The trick is now to differentiate these relations, which yields
  \begin{align*} 
    f'(t) = ct\left( 2(t^2+at+b) + t(2t+a)\right) &= c(t-1)^2\left( (t-1) +
    3(t+d)\right) \\
    t(4t^2+3at+2b) &= (t-1)^2\left(4t+(3d-1)\right).
  \end{align*}
  By unique factorization, we must have $4t^2+3at+2b=4(t-1)^2$ and $4t =
  4t+(3d-1)$, so we instantly get $a=-8/3$, $b=2$, and $d=1/3$.  Substituting
  back we see that $c=3$, and obtain
  \begin{align*}
    f(t) = t^2(3t^2-8t+6) = (t-1)^3(3t+1) + 1. 
  \end{align*}
\end{exm}

More generally, the differentiation trick is an observation on divisors that
extends to higher genus, as used by Elkies \cite{Elkies237} in genus $g=1$. 

\begin{lem} \label{lem:belyiram}
  Let $f: X \to \PP^1$ be a \Belyi\ map with ramification type $\sigma$.  Let
  \begin{align*}
    \opdiv f = \sum_P e_P P - \sum_R e_R R 
    \quad \text{and} \quad
    \opdiv(f-1) = \sum_Q e_Q Q - \sum_R e_R R 
  \end{align*}
  be the divisors of $f$ and $f-1$.  Then the divisor of the differential $df$
  is
  \begin{align*}
    \opdiv df = \sum_P (e_P-1) P + \sum_Q (e_Q-1) Q - \sum_R (e_R+1) R. 
  \end{align*}
\end{lem}
\begin{proof}
  Let
  \begin{align*}
    D= \sum_P (e_P-1) P + \sum_Q (e_Q-1) Q - \sum_{R} (e_R+1)R. 
  \end{align*}
  Then $\opdiv df \geq D$ by the Leibniz rule.  By Riemann--Hurwitz, we have
  \begin{align*}
    2g-2=-2n+\sum_{P} (e_P-1) + \sum_{Q} (e_Q-1) + \sum_{R} (e_R-1) 
  \end{align*}
  so 
  \begin{align*}
    \deg(D)=2g-2+2n-2\sum_R e_R = 2g-2 
  \end{align*}
  since $\sum_R e_R = n$.  Therefore $\opdiv df$ can have no further zeros.
\end{proof}

Combined with unique factorization, this gives the following general algorithm
in genus $0$. Write
\begin{align*}
  f(t) = \frac{p(t)}{q(t)} = 1 + \frac{r(t)}{q(t)}
\end{align*}
for polynomials $p(t),q(t),r(t) \in \Qbar [t]$. Consider the derivatives
$p'(t),q'(t),r'(t)$ with respect to $t$ and let $p_0(t)=\gcd(p(t),p'(t))$ and
similarly $q_0(t),r_0(t)$.  Write
\begin{align*}
  P(t) = \frac{p(t)}{p_0(t)} \text{ and } \widetilde{P}(t) =
  \frac{p'(t)}{p_0(t)}
\end{align*}
and similarly $Q$, etc.  Then by unique factorization, and the fact that $P , Q
, R$ have no common divisor, evaluation of the expressions $p(t) - q(t) = r(t)$
and $p'(t) - q'(t) = r'(t)$ yields that $Q(t) \widetilde{R}(t) -
\widetilde{Q}(t) R(t)$ is a multiple of $p_0(t)$, and similarly $P(t)
\widetilde{R}(t) - \widetilde{P}(t)R(t)$ (resp.\ $P(t) \widetilde{Q}(t) -
\widetilde{P}(t) Q(t)$) is a multiple of $q_0(t)$ (resp.\  $r_0(t)$).

These statements generalize to higher genus, where they translate to inclusions
of divisors; but the usefulness of this for concrete calculations is limited and
do not pass to relations of functions, since the coordinate rings of higher
genus curves are usually not UFDs. Essentially, one has to be in an especially
agreeable situation for a statement on functions to fall out, and usually one
only has a relation on the Jacobian (after taking divisors, as in the lemma
above). A concrete and important situation where a relation involving functions
does occur is considered by Elkies \cite{Elkies237}. The methods in his example
generalize to arbitrary situations where the ramification is uniform (all
ramification indices equal) except at one point of the \Belyi\ curve: Elkies himself treats the \Belyi\ maps with
passport $(1, \PSL_2 (\F_{27}), (3^9 1^1, 2^{14}, 7^4))$.

The differentiation trick does not seem to generalize extraordinarily well to
higher derivatives; we can repeat the procedure above and further differentiate
$p'(t),q'(t),r'(t)$, but experimentally this not seem to make the ideal grow
further than in the first step. 

\begin{ques}
  Is the ideal obtained by adding all higher order derivatives equal to the one
  obtained from just adding equations coming from first order derivatives (in
  genus $0$)?
\end{ques}

However, Shabat \cite[Theorem 4.4]{Shabat} does derive some further information
by considering second-order differentials; and Dremov \cite{Dremov}
calculates \Belyi\ maps using the quadratic differential
\begin{align*}
  MP(f) = \frac{df^2}{f(1 - f)}
\end{align*}
for a regular function $f$ and considering the equalities following from the
relation 
\[ MP(f^{-1}) = -MP(f) / f. \] 
 It is not immediately clear from these
paper how to use this strategy in general, though.

\begin{ques}
  How generally does the method of considering second-order differentials apply?
\end{ques}

The additional equations coming from the differentiation trick not only speed up
the process of calculating \Belyi\ maps, but they also tend to give rise to a
Jacobian matrix at a solution that is often of larger rank than the direct
system. This is important when trying to Hensel lift a solution obtained over
$\C$ or over a finite field, where the non-singularity of the Jacobian involved
is essential.  (We discuss these methods in sections that follow.)

\begin{rmk}\label{rmk:naive2}
  Phrased in the language of the naive moduli space in Remark \ref{rmk:naive1},
  the additional ASD relations partially saturate the corresponding equation
  ideal, so that the larger set of equations defines the same set of geometric
  points, but with smaller multiplicities.  (We thank Bernd Sturmfels for this
  remark.)  Reducing this multiplicity all the way to $1$ is exactly the same as
  giving the Jacobian mentioned above full rank.
\end{rmk}

\begin{exm}
  The use of this trick for reducing multiplicities is best illustrated by some
  small examples.

  The first degree $d$ in which the ASD differentiation trick helps to give the
  Jacobian matrix full rank is $d=6$; it occurs for the ramification triples
  $(2^3,2^3,3^2)$, $(2^2 1^2, 3^2, 4^1  2^1 )$, $(3^2, 3^1  2^1  1, 3^1  2^1
  1^1)$, $(3^1 1^3, 4^1 2^1, 4^1  2^1)$, $(4^1  2^1, 4^1  1^2, 3^1  2^1  1^1)$,
  and $(4^1  2^1, 3^1  2^1  1^1, 3^1 2^1  1^1)$, where it reduces the
  multiplicity of the corresponding solutions from $9,3,3,3,4,3$ respectively to
  $1$. Note the tendency of \Belyi\ maps with many automorphisms to give rise to
  highly singular points, as for curves with many automorphisms in the
  corresponding moduli spaces.

  On the other hand, there are examples where even adding the ASD relations does
  not lead to a matrix of full rank. Such a case is first found in degree $7$;
  it corresponds to the ramification triples $(4^1  2^1  1^1, 3^1  2^1  1^2 ,
  4^1 3^1)$, and throwing in the ASD relations reduces the multiplicity from $8$
  to $2$.  Unfortunately, iterating the trick does not make the ideal grow
  further in this case.

  More dramatically, for the ramification triples $(2^4, 3^2 2^1, 3^2
  2^1)$ and $(2^3 1^2, 4^2 , 3^2 2^1)$, differentiation reduces some
  multiplicities from $64$ to $1$ (resp.\ $64$ to $4$). In the latter case,
  these multiplicities are in fact not determined uniquely by the corresponding
  ramification type, so that considering these multiplicities gives a way to
  split the solutions into disjoint Galois orbits.
\end{exm}

\begin{ques}
  How close is the ideal obtained from the differentiation trick (combined with
  the direct method) to being radical?  Can one give an upper bound for the
  multiplicity of isolated points?
\end{ques}

\subsection*{Further extensions}

There can be several reasons why a Gr\"obner basis calculation fails to
terminate.  One problem is coefficient blowup while calculating the elimination
ideals.  This can be dealt by first reducing modulo a suitable prime $p$,
calculating a Gr\"obner basis for the system modulo $p$, then lifting the good
solutions (or the Gr\"obner basis itself) $p$-adically, recognizing the
coefficients as rational numbers, and then verifying that the basis over $\Q$ is
correct.  This was used by Malle \cite{Malle15,MalleTrinks} to compute covers
with passports $(0 , \Hol(E_8) , (4^1 2^1 1^2,4^1 2^1 1^2,6^1 2^1))$ and $(0 ,
\PGL(\F_{11}) , (2^5 1^2,4^3,11^1 1^1))$ and similarly Malle--Matzat
\cite{MalleMatzatPSL2Fp} to compute covers for $(0 ,
\PSL_2(\F_{11}) , (2^4 1^3, 6^1 3^1 2^1, 6^1 3^1 2^1))$ and $(0 ,
\PSL_2(\F_{13}) , (2^7, 4^3 1^2, 6^2 1^2))$.  This idea was also used by
Vidunas--Kitaev \cite[\S 5]{VK}.  For further developments on $p$-adic methods
to compute Gr\"obner bases, see Arnold \cite{Arnold} or Winkler \cite{Winkler}.
One can also lift a solution modulo $p$ directly, and sometimes such solutions
can be obtained relatively quickly without also $p$-adically lifting the
Gr\"obner bases: this is the basic idea presented in Section \ref{sec:padic}.

In the work of van Hoeij--Vidunas \cite{vHV2,vHV} mentioned in Section
\ref{sec:backg}, genus $0$ \Belyi\ functions are computed by using pullbacks of
the hypergeometric differential equation and their solutions.  This method works
well when the order of each ramification point is as large as possible, e.g.,
when the permutations $\sigma_0,\sigma_1,\sigma_{\infty}$ contain (almost) solely
cycles of order $n_0,n_1,n_{\infty}$ say, and only a few cycles of smaller
order. For example, this occurs when the cover is Galois, or slightly weaker,
when it is \defi{regular}, that is to say, when the permutations
$\sigma_0,\sigma_1,\sigma_\infty$ are a product of disjoint cycles of equal
cardinality.

The method of van Hoeij--Vidunas to calculate a \Belyi\ map $f : X \to \PP^1$ is
to consider the $n$ \defi{exceptional} ramification points in $X$ of $f$ whose
ramification orders do not equal the usual orders $a,b,c$. One then equips the
base space $\PP^1$ with the hypergeometric equation whose local exponents at
$0,1,\infty$ equal $a,b,c$. Pulling back the hypergeometric equation by $f$, one
obtains a Fuchsian differential equation with singularities exactly in the $n$
exceptional points. The mere fact that this pullback exists implies equations on
the undetermined coefficients of $f$.

For example, when the number of exceptional points is just $n=3$, the
differential equation can be renormalized to a Gaussian hypergeometric
differential equation, which completely determines it. When $n=4$, one obtains a
form of Heun's equation \cite{MovasatiReiter, vHV2}. Heun's equation depends on
the relative position of the fourth ramification point, as well as on an
\defi{accessory parameter}; still, there are only two parameters remaining in
the computation.

One shows that for fixed $n$ and genus $g$ (taken as $g=0$ later), there are
only finitely many hyperbolic \Belyi\ functions with $n$ exceptional points.
For small $n$, van Hoeij and Vidunas show that this differential method is
successful in practice, and they compute all (hyperbolic) examples with $n \leq
4$ (the largest degree of such a \Belyi\ map was $60$).  

\begin{ques}
  Are there other sources of equations (such as those arising from differential
  equations, algebraic manipulation, etc.)\ that further simplify the
  scheme obtained from the direct method?
\end{ques}

\section{Complex analytic methods}\label{sec:coan}

In this section we consider complex analytic methods for finding equations for
\Belyi\ maps. These methods are essentially approximative; a high precision
solution over $\C$ is determined, from which one reconstructs an exact
solution over $\Qbar$.

\subsection*{Newton approximation}

We have seen in the previous section how to write down a system of equations
which give rise to the \Belyi\ map.  These equations can be solved numerically
in $\C$ using multidimensional Newton iteration, given an approximate solution
that is correct to a sufficient degree of precision and a subset of equations of
full rank whose Jacobian has a good condition number (determinant bounded away
from zero).  Then, given a complex approximation that is correct to high
precision, one can then use the LLL lattice-reduction algorithm \cite{LLL} (as 
well as other methods, such as PSLQ \cite{PSLQ}) to guess algebraic numbers that
represent the exact values. Finally, one can use the results from Section
\ref{sec:veri} to verify that the guessed cover is correct; if not, one can go
back and iterate to refine the solution.

\begin{rmk}\label{rem:algdep}
  We may repeat this computation for each representative of the Galois orbit to
  find the full set of conjugates for each putative algebraic number and then
  recognize the symmetric functions of these conjugates as rational numbers
  using continued fractions instead.  For example, one can compute each
  representative in the passport, possibly including several Galois orbits.  The
  use of continued fractions has the potential to significantly reduce the
  precision required to recognize the \Belyi\ map exactly.
\end{rmk}

\begin{exm}
  Consider the permutation triple
  \begin{align*}
    \sigma_0 = (1\ 3\ 2)(4\ 6\ 5), \quad
    \sigma_1 = (1\ 5\ 2)(3\ 4)(6\ 7), \quad
    \sigma_{\infty}=(1\ 3\ 5\ 2\ 6\ 7\ 4).
  \end{align*}
  From the Riemann--Hurwitz formula, we find that the associated \Belyi\ curve
  $X$ has genus $g=1$.  The ramification point of index $7$ on $X$ (over
  $\infty$) is unique, so we take it to be the origin of the group law on $X$.
  Moreover, since there is a unique unramified point above $0$, we can use a
  normal form (due to Tate) of an elliptic curve with a marked point.  This is
  given by an equation
  \begin{align}\label{eq:M11}
    y^2 + p_3 y = q(x) =  x^3 + p_2 x^2 + p_4 x 
  \end{align}
  with marked point $(0,0)$. The equation \eqref{eq:M11} is unique up to scaling
  the coefficients by $u\neq 0$ according to $(p_2 ,
  p_3 , p_4 ) \mapsto (u^2 a_2 , u^3 a_3 , u^4 a_4)$, showing that the moduli
  spaces $\calM_{1,2}$ of genus $1$ curves with two marked points is isomorphic
  to the weighted projective space $\PP (2,3,4)$. 

  Since the origin of the group law of $X$ maps to $\infty$ and $(0,0)$ maps to $0$,
  the \Belyi\ map $f:X \to \PP^1$ of degree $7$ is of the form
  \begin{align*}
    f(x,y) = (a_3x^3+a_2x^2+a_1x) + (b_2x^2+b_1x+b_0)y = a(x)+b(x)y.
  \end{align*}
  The ramification above $f=0$ leads to the equation
  \begin{align*}
    \N_{\C (x,y) / \C (x)} ( f(x,y) ) = a(x) (a(x)- p_3 b(x)) - b(x)^2 q(x) =
    -b_2^2 x c(x)^3
  \end{align*}
  where $c(x)$ is the monic polynomial $x^2+c_1x+c_0$. Consideration of the
  ramification above $1$ yields
  \begin{align*}
    \N_{\C (x,y) / \C (x)} ( f(x,y) - 1 ) = (a(x) - 1) (a(x) - 1 - p_3 b(x)) -
    b(x)^2 q(x) = -b_2^2 d(x)^3 e(x)^2
  \end{align*}
  where $d(x)=x+d_0$ and $e(x)=x^2+e_1x+e_0$.
  
  This yields $13$ equations in $14$ unknowns. The reason for this is that we are
  still free to scale the $p_i$. Here we have to distinguish cases. We first
  suppose that the point $(0,0)$ in (\ref{eq:M11}) is not $2$-torsion, or equivalently,
   that $p_3 \neq 0$: this is the ``generic'' case.  We can then distinguish two further
  cases, namely $p_2 \neq 0$ and $p_4 \neq 0$. 
  Accordingly, we may then ensure $p_2=p_3$ or $p_3=p_4$ by scaling over the ground
  field, so that we do not needlessly
  enlarge the coefficient of the \Belyi\ map.  In either case, plugging in random choices
  for the vector of unknowns $(a,b,c,d,e,p) \in \C^{14}$  
  and applying multivariate Newton iteration fails to yield a solution. 

  To improve the convergence, we now proceed to remove some degenerate cases from
  this set of equations. Applying the trick from Example \ref{exm:largegrob}, we
  impose that $c_0 d_0 e_0 \neq 0$, as we may since the ramification points are
  distinct and $(0,0)$ is a ramification point.  
  (This in fact assumes that none
  of the other ramification points is $(0,-1)$, which leads to a
  subcase that turns out not to yield a solution.)  
  We further insist that $c$ and $e$ do not
  have a double root, so $(c_1^2-c_0)(e_1^2-e_0) \neq 0$. This adds $2$
  more variables and equations.

  Finally, we saturate our equations using the Atkin--Swinnerton-Dyer trick in 
  Lemma \ref{lem:belyiram}. The
  differential $dx/(2 y + p_3)$ is holomorphic and has no zeros or poles, so
  denoting derivation with respect to $x$ by $'$, we see that
  \begin{align*}
    \frac{df}{dx/(2 y + p_3)} = (2 y + p_3) \frac{df}{dx} & = a'(x) (2 y + p_3)
    + b'(x) (2 y + p_3) y + b(x) (2 y + a_3) y' \\
    & = (2 b'(x) q (x) + b (x) q'(x) + p_3 a'(x)) + (2 a'(x) - p_3 b'(x)) y
  \end{align*}
  satisfies
  \begin{align*}
    \N(((2 y + p_3) (df/dx)) = 49 b_2^2 c(x)^2d(x)^2e(x)  .
  \end{align*}
  This differentation trick thus yields another $8$ equations. But even after
  adding these and the nondegeneracy conditions, random choices for an initial
  approximation fail to converge to a solution for the new system of $23$
  equations in $15$ unknowns.

  So we are led to consider the case where $p_3 = 0$, so that the unramified point above
  $0$ is $2$-torsion. (Here, there is some extra ambiguity, since the moduli space $X_0(2)=X_1(2)$ 
  is not a fine moduli space.)  If we write
  $\opdiv(f)=(0,0)+3P_1+3P_2-7\infty$ and $\opdiv(f-1)=2Q_1+2Q_2+3Q_3$, then we
  have
  \begin{align*}
    \opdiv(df)=2P_1+2P_2+Q_1+Q_2+2Q_3-8\infty
  \end{align*}
  and so we obtain the relations 
  \begin{align*}
    Q_1+Q_2=3Q_3=0, \quad 3Q_3=0, \quad (0,0)+(P_1+P_2)=-Q_3, \quad 2(0,0)=0
  \end{align*}
  in the group law of $X$. In particular, $P_1+P_2$ is a $6$-torsion point on
  $X$. Relations such as these can be used to find extra equations for $X$ and
  $f$ by using division polynomials. But again, the new system fails to yield
  any solutions; perhaps one can prove non-existence of solutions directly.
  
  Here, we look ahead to the methods of this section and 
  Section \ref{sec:mod} that allow us to find an approximation
  to the solution. It turns out that we only need $3$ decimal places to get the
  Newton method converging to a real solution with $p_3 \neq 0$ and $p_2 = p_3$,
  approximated by the solution
  \begin{align*}
    &(a,b,c,d,e,p) \approx \\
    &\quad ( 182.7513294, 146.8290694, 29.38993410, -308.3482399, -244.0552479 \\
    &\qquad -48.11742858, 0.7992141684, 0.1613326212, 0.1482181605, 0.9764940118, \\
    &\qquad 0.2561882114, 1.165925608, 0.4430649844, 163.2364906, 3.003693522 )
  \end{align*}
  in $\C^{13}$.  The condition number of the system without the additional
  Atkin--Swinnerton-Dyer relations is approximately $3.3 \cdot 10^7$; but by adding
  some of these relations, this can be decreased to approximately $1.2 \cdot 10^5$.

  Using LLL, we recognize this as a putative 
  solution over $\Q (\alpha)$ with $\alpha^3 - 3 \alpha + 12 = 0$; then
  we verify that the recognized solution is correct using 
  the methods of Section \ref{sec:veri}.  This solution thereby gives rise
  to two more complex (conjugate) solutions. Since there are only three permutation triples
  with the given ramification passport, we see that we have found all dessins of
  the given ramification type, so we need not consider the other cases further.

  As mentioned in Remark \ref{rem:algdep}, the standard algorithms to recognize
  algebraic dependency work better after symmetrizing over these conjugate
  solutions. For the most difficult algebraic number to recognize (which is
  $b_2$) using a single solution requires the knowledge of $161$ digits, whereas
  recognition as an algebraic number needs only $76$ digits.

  If we drop the demand that the unramified point is at $(0,0)$, then we can
  simplify the solution somewhat, as in Section \ref{sec:veri}. In Weierstrass
  form, we can take $X$ to be given by the curve
  \begin{align*}
    y^2 = x^3 & + (-541809 \alpha^2 + 898452 \alpha + 2255040) x \\
    & + (-2929526838 \alpha^2 + 5759667648 \alpha - 11423888784)  .
  \end{align*}
  and the function $f = a(x) + b (x) y$ by
  \begin{align*}
    2^{13} 3^{14} 5^5 a(x) & = (1491 \alpha^2 + 6902 \alpha + 10360) x^3 \\
    & + (1410885 \alpha^2 + 2033262 \alpha - 4313736) x^2 \\
    & + (731506545 \alpha^2 + 15899218650 \alpha + 32119846920) x \\
    & - (7127713852353 \alpha^2 + 3819943520226 \alpha + 62260261739784)
  \end{align*}
  and
  \begin{align*}
    2^{13} 3^{16} 5^5 b(x) & = (-197 \alpha^2 - 240 \alpha + 528) x^2 \\
    & + (906570 \alpha^2 - 546840 \alpha - 8285760) x \\
    & - (715988241 \alpha^2 - 2506621464 \alpha - 1458270864)  .
  \end{align*}
  We thank Marco Streng for his help with reducing these solutions. 
  Applying the methods in Section \ref{sec:mod} already gives
  equations that are better than those in the normalized forms (\ref{eq:M11})
  considered above; at least experimentally, using the modular method also tends to
  give equations of relatively small height.
\end{exm}

As we have seen in the preceding example, 
in order for this procedure to work, one needs a good starting approximation to
the solution.  In the non-trivial examples that we have computed so far, it
seems that often this approximation must be given to reasonably high precision
(at least 30 digits for moderately-sized examples) in order for the convergence
to kick in. The required precision seems difficult to estimate from above or
below.  And indeed the dynamical system arising from Newton's method has quite
delicate fractal-like properties and its study is a subject in itself
\cite{NewtonDynam}.

\begin{ques}
  Is there an explicit sequence of \Belyi\ maps with the property that the
  precision required for Newton iteration to converge tends to infinity?
\end{ques}

One way to find a starting approximation to the solution is explained by
Couveignes--Granboulan
\cite{CouveignesCRF,GranboulanThesis,CouveignesGranboulan}.  They inductively
use the solution obtained from a simpler map: roughly speaking, they replace a
point of multiplicity $\nu$ with two points of multiplicities $\nu_1,\nu_2$ with
$\nu_1+\nu_2=\nu$.  One can use any appropriate base case for the induction,
such as a map having simple ramification.  Couveignes \cite{CouveignesCRF} gives
a detailed treatment of the case of \defi{trees}, corresponding to clean \Belyi\
polynomials $f(t)$, i.e.\ those with $f(t)-1 = g(t)^2$: geometrically, this
means that the corresponding dessin can be interpreted as a tree with oriented
edges.  In this case, after an application of the differentiation trick, one is
led to solve a system of equations where many equations are linear. See
Granboulan \cite[Chapter IV]{GranboulanThesis} for an example with monodromy
group $\Aut(M_{22})$. 

\begin{rmk}
  There is a misprint in the example of Couveignes \cite[\S 3, pg.\
  8]{CouveignesCRF} concerning the discriminant of the field involved, corrected
  by Granboulan \cite[p.\ 64]{GranboulanThesis}.
\end{rmk}

So far, it seems that the inductive numerical method has been limited to genus
$0$ \Belyi\ maps with special features.  A similar method was employed by
Matiyasevich \cite{Matiyasevich} for trees: he recursively transforms the
initial polynomial $2t^n-1$ (corresponding to a star tree) into a polynomial
representing the desired planar tree.

\begin{ques}
  Can an inductive complex analytic method be employed to compute more
  complicated \Belyi\ maps in practice?
\end{ques}

In particular, the iterative method by Couveignes and Granboulan to find a good
starting value seems to rely on intuition involving visual considerations; can
these be made algorithmically precise?

\subsection*{Circle packing}

Another complex analytic approach is to use \defi{circle packing methods}. This
technique was extensively developed in work of Bowers--Stephenson
\cite{BowersStephenson}, with a corresponding Java script \texttt{CirclePack}
available for calculations.

Given a dessin (i.e., the topological data underlying a \Belyi\ map), one
obtains a triangulation of the underlying surface by taking the inverse image of
$\PP^1 (\R) \subset \PP^1 (\C)$ together with the corresponding cell
decomposition. Choosing isomorphisms between these triangle and the standard
equilateral triangle in $\C$ and gluing appropriately, one recovers the Riemann
surface structure and as a result a meromorphic description of the \Belyi\ map.

However, the Riemann surface structure is difficult to determine explicitly,
starting from the dessin.  As an alternative, one can pass to \defi{discrete
\Belyi\ maps} instead.  To motivate this construction, note that a Riemann
surface structure on a compact surface induces a unique metric of constant
curvature $1,0,-1$ (according as $g=0,1,\geq 2$) so that one can then speak
meaningfully about circles on such a surface. In particular, it makes sense to
ask whether or not there exists a \defi{circle packing} associated with the
triangulation, a pattern of circles centered at the vertices of this
triangulation satisfying the tangency condition suggested by the triangulation.
Satisfyingly enough, the \defi{circle packing theorem}, due
Koebe--Andreev--Thurston \cite{Koebe,MardenRodin,Thurston}, states that given a
triangulation of a topological surface, there exists a unique structure of
Riemann surface that leads to a compatible circle packing. This then realizes
the topological map to the Riemann sphere as a smooth function.

In summary, starting with a dessin, one obtains a triangulation and hence a
circle packing.  The corresponding discrete \Belyi\ map will in general
\emph{not} be meromorphic for the Riemann surface structure induced by the
circle packing; but Bowers and Stephenson prove that it does converge to the
correct solution as the triangulation is iteratively hexagonally refined.

The crucial point is now to compute the discrete approximations obtained by
circle packing in an explicit and efficient way. Fortunately, this is indeed
possible; work by Collins--Stephenson \cite{CollinsStephenson} and Mohar
\cite{Mohar} give algorithms for this. The crucial step is to lift the
configuration of circles to the universal cover $H$ (which is either the sphere
$\PP^1 (\C)$, the plane $\C$, or the upper half-plane $\calH$) and perform the
calculation in $H$. In fact, this means that the circle packing method also
explicitly solves the \defi{uniformization problem} for the surface involved;
for theoretical aspects, we refer to Beardon--Stephenson
\cite{BeardonStephenson}.  Upon passing to $H$ and using the appropriate
geometry, one then \emph{first} calculates the radii of the circles involved
from the combinatorics, before fitting the result into $H$, where it gives rise
to a fundamental domain for the corresponding curve as a quotient of $H$.  

An assortment of examples of the circle packing method is given by
Bowers--Stephenson \cite[\S 5]{BowersStephenson}, and numerical approximations
are computed to a few digits of accuracy.  This includes genus $0$ examples of
degree up to $18$, genus $1$ examples of degree up to $24$, and genus $2$
examples of degree up to $14$. For determining the conformal structure, this
approach is therefore much more effective indeed than the naive method from
Section \ref{sec:groeb}. Even better, one can proceed inductively from simpler
dessins by using so-called \defi{dessin moves} \cite[\S 6.1]{BowersStephenson},
which makes this approach quite suitable for calculating large tables of
conformal realizations of dessins.

On the other hand, there are no theoretical results on the number of refinements
needed to obtain given accuracy for the circle packing method \cite[\S
7]{BowersStephenson}.  In examples, it is possible for the insertion of a new
vertex to drastically increase the accuracy needed \cite[Figure
25]{BowersStephenson} and thereby the number of discrete refinements needed,
quite radically increasing the complexity of the calculation \cite[\S
8.2]{BowersStephenson}.  However, the method is quite effective in practice,
particularly in genus $0$.

More problematically, it seems difficult to recover equations over $\Qbar$ for
the \Belyi\ map from the computed fundamental domain if the genus is strictly
positive.  One can compute the periods of the associated Riemann surface to some
accuracy, but one still needs to recover the curve $X$ and transfer the \Belyi\
map $f$ on $X$ accordingly.  Moreover, is also not clear that the accuracy
obtained using this method is enough to jump start Newton iteration and thereby
obtain the high accuracy needed to recognize the map over $\Qbar$.  In Section
\ref{sec:mod}, we circumvent this problem by starting straightaway with an
explicit group $\Gamma$ of isometries of $H$ so that $\Gamma \backslash H \cong
X$ and then finding equations for $X$ by numerically computing modular forms
(i.e., differential forms) on $X$.

\begin{exm}
  In Figure \ref{fig:m23}, we give an example from an alternate implementation
  by Westbury, which is freely available \cite{Westbury} for the case of
  genus $0$. In the figure, an outer polygon is inserted instead of a circle to
  simplify the calculation of the radii. We show the conformal triangulation
  induced by the second barycentric subdivision of the original triangulation
  for one of the exactly $2$ covers in Example \ref{ex:mathest} that descend to
  $\R$.

  \begin{equation} \label{fig:m23} \notag
    \includegraphics[scale=0.8]{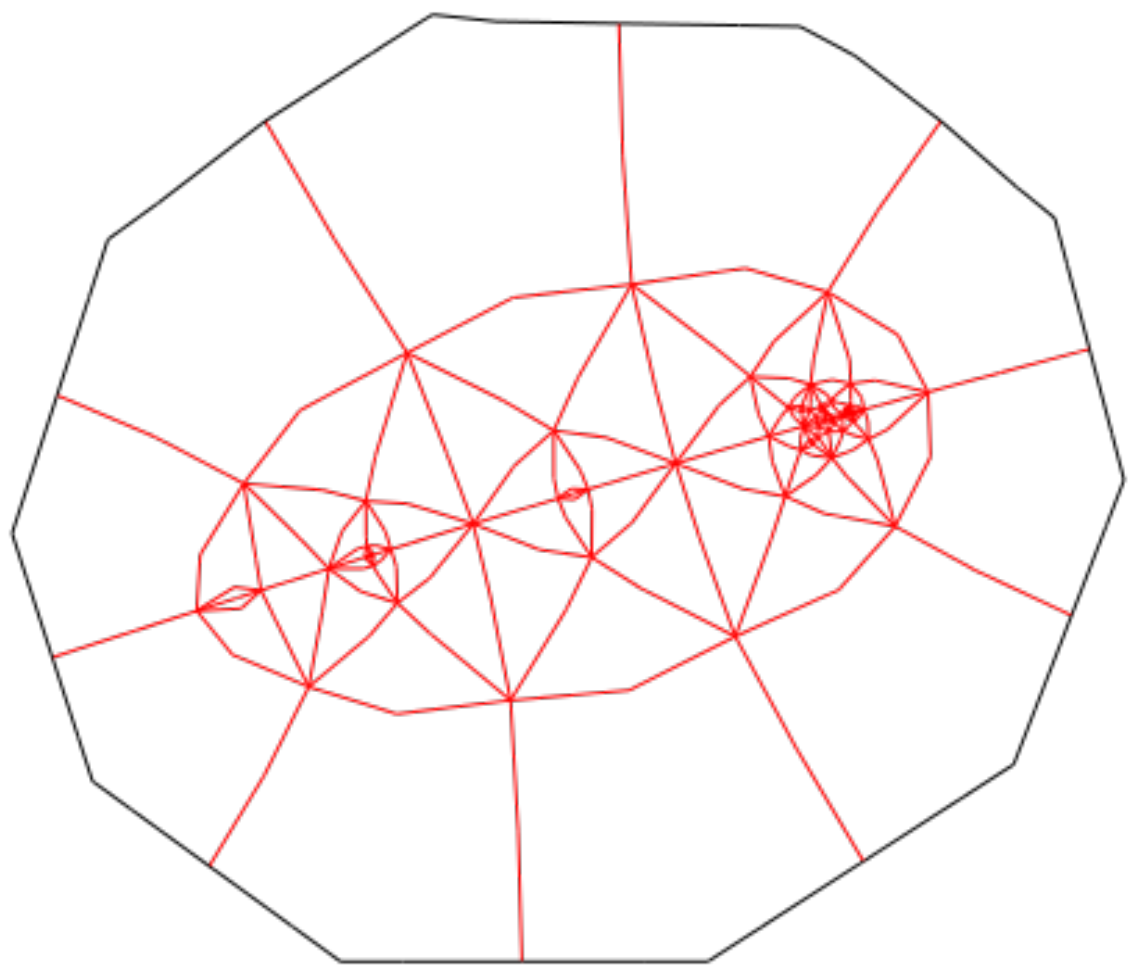}
  \end{equation}
  \begin{center}
  \textbf{Figure \ref{fig:m23}}: A second subdivision for $M_{23}$ \\
  \vspace{2ex}
  \end{center}
  \addtocounter{equation}{1}

  Several more subdivisions would be needed to get the solution close enough to
  apply Newton's method.
\end{exm}

\subsection*{Puiseux series}

Couveignes--Granboulan \cite[\S 6]{CouveignesGranboulan} proposed an alternative
method using Puiseux series expansions to get a good complex approximation to
the solution so that again multidimensional Newton iteration can kick off.

At every regular point $P$ in the curve $X$, the \Belyi\ map has an analytic
expansion as a power series in a uniformizer $z$ at $P$ that converges in a
neighborhood of $P$.  Similarly, at a ramification point $P$, there is an
expansion for $f$ that is a Puiseux series in the uniformizer $z$; more
specifically, it is a power series in $z^{1/e}=\exp(2\pi i \log(z)/e)$ where $e$
is the ramification index of $P$ and $\log$ is taken to be the principal
logarithm.  Now, these series expansions must agree whenever they overlap, and
these relations between the various expansions give conditions on their
coefficients.  More precisely, one chooses tangential base points, called
\defi{standards}, and the implied symbolic relations are then integrated with
respect to a measure with compact support. Collecting the relations, one obtains
a block matrix, the positioning of whose blocks reflects the topology of the
overlaps of the cover used. 

Unfortunately, Couveignes and Granboulan do not give an example of this method
in practice, and the most detail they give concerns iterative ad hoc methods
\cite[\S 7]{CouveignesGranboulan}. 

\begin{ques}
  How effective is the method of Puiseux series in finding a good starting
  approximation?  Can one prove rigorously that this method gives a correct
  answer to a desired precision?
\end{ques}

\subsection*{Homotopy methods}

One idea that has yet to be explored (to the authors' knowledge) is the use of
techniques from numerical algebraic geometry, such as polyhedral homotopy
methods \cite{Bertini,Verschelde}, to compute \Belyi\ maps.  The success of
homotopy methods in solving extremely large systems of equations, including
those with positive-dimensional components, has been dramatic.  In broad
stroke, one deforms the solution of an easier system to the desired ones and
carefully analyzes the behavior of the transition matrix (Jacobian) to ensure
convergence of the final solution.  Because these methods are similar in spirit
to the ones above, but applied for a more general purpose, it is natural to
wonder if these ideas can be specialized and then combined into a refined
technique tailored for \Belyi\ maps.

\begin{ques}
  Can the techniques of numerical algebraic geometry be used to compute \Belyi\
  maps efficiently?  
\end{ques}

A potential place to start in deforming is suggested by the work above and by
Couveignes \cite[\S 6]{CouveignesTools}: begin with a stable curve (separating
the branch points) and degenerate by bringing together the genus 0 components.
The difficulty then becomes understanding the combinatorial geometry of this
stable curve, which is an active area of research.

\subsection*{Zipper method}

Complex analytic techniques can also be brought to bear on \Belyi\ maps of
extremely large degree, at least for the case of trees, using an extension of
the \defi{zipper method} due to Marshall--Rohde \cite{Zipper1,Zipper2}.  The
zipper method finds a numerical approximation of the conformal map of the unit
disk onto any Jordan region \cite{Marshallweb}.  In its extension, this amount
to solving the Dirichlet problem with boundary for the domain of the exterior of
the desired dessin, which can be done quite simply for trees even with thousands
of branches.  For example, Marshall and Rohde have computed the dessins
associated to the \Belyi\ maps $f^n(z)$ where $f(z)=(3z^3-9z-2)/4$, giving a
sequence of \Belyi\ trees (under the preimage of $[-2,1]$), and by extension one
can obtain complex approximation to \Belyi\ maps of extremely large degree:
trees with tens of thousands of edges, far beyond the reach of other methods.

\begin{ques}
  Does the zipper method extend to higher genus?
\end{ques}

In the latter extension, one would need to consider not only the convergence of
the \Belyi\ map but also the associated \Belyi\ curve $X$, so it appears one will
have to do more than simply solve the Dirichlet problem.  See also work by
Larusson and Sadykov \cite{LarussonSadykov}, where the connection with the
classical Riemann-Hilbert problem is discussed in the context of trees.

\section{Modular forms}\label{sec:mod}

In this section we continue with the general strategy of using complex analytic
methods but shift our focus in the direction of geometry and consideration of
the uniformization theorem; we work explicitly with quotients of the upper
half-plane by Fuchsian groups and recast \Belyi\ maps in this language.   This
point of view is already suggested by Grothendieck \cite{Grothendieck}: 
\begin{quote}
  In more erudite terms, could it be true that every projective non-singular
  algebraic curve defined over a number field occurs as a possible ``modular
  curve'' parametrising elliptic curves equipped with a suitable rigidification?
  \ldots [T]he Soviet mathematician \Belyi\ announced exactly that result.
\end{quote}

As in the last section, the method here uses numerical approximations; however,
the use of modular functions adds considerable more number-theoretic flavor to
the analytic techniques in the previous section.

\subsection*{Classical modular forms}

Let $F_2$ be the free group on two generators as in \eqref{eqn:F2}.  Recall that
the map that considers the permutation action of $x,y,z$ on the cosets of a
subgroup yields a bijection
\begin{equation}\label{eq:freecorresp}
  \begin{gathered} 
    \left\{ \text{transitive permutation triples $\sigma = (\sigma_0, \sigma_1,
      \sigma_{\infty}) \in S_d^3$} \right\} /\!\sim
    \\
    \updownarrow \text{\small{1:1}}
    \\
    \left\{ \text{subgroups of $F_2$ of index $d$} \right\} /\!\sim\,;
  \end{gathered}
\end{equation}
here the equivalence relation on triples is again uniform conjugation, and the
equivalence relation on subgroups is conjugation in $F_2$.  In particular, by
Proposition \ref{prop:cateq2}, isomorphism classes of (connected) \Belyi\ maps
are in bijection with the conjugacy classes of subgroups $F_2$ of finite index.

The key observation is now that $F_2$ can be realized as an arithmetic group,
as follows.  
The group $\Gamma(1)=\PSL_2(\Z)=\SL_2(\Z)/\{\pm 1\}$ acts on the completed upper
half-plane $\calH^*=\calH \cup \PP^1(\Q)$ by linear fractional transformations
\begin{align*}
  z \mapsto \frac{az+b}{cz+d}, \quad \text{for}\ \pm \begin{pmatrix} a & b \\ c
    & d \end{pmatrix} \in \PSL_2(\Z).
\end{align*}
The quotient $X(1)=\Gamma(1)\backslash\calH^*$ can be given the structure of a
Riemann surface of genus $0$ by the uniformizing map $j: X(1) \xrightarrow{\sim}
\PP^1(\C)$ (often called the \defi{modular elliptic $j$-function}),
\begin{align*}
  j(q)=\frac{1}{q}+744+196884q+21493760q^2+864299970q^3+\dots
\end{align*}
where $q=\exp(2\pi iz)$.  

For an integer $N$, we define the normal subgroup $\Gamma(N)$ as the kernel of
the reduction map $\PSL_2(\Z) \to \PSL_2(\Z/N\Z)$.  We will be particularly
interested in the subgroup 
\begin{align*}
  \Gamma(2)= \left\{ \pm \begin{pmatrix} a & b \\ c & d \end{pmatrix} \in
  \PSL_2(\Z) : b \equiv c \equiv 0 \psmod{2} \right\}
\end{align*}
of index $6$, with quotient isomorphic to $\Gamma(1)/\Gamma(2) \cong
\PSL_2(\F_2)=\GL_2(\F_2) \cong S_3$.  The group $\Gamma(2)$ is in fact
isomorphic to the free group $F_2 \cong \Gamma(2)$: it is freely generated by
 $\pm \begin{pmatrix} 1 & 2 \\ 0 & 1 \end{pmatrix},\pm
\begin{pmatrix} 1 & 0 \\ 2 & 1 \end{pmatrix}$, which act on $\calH$ by $z
\mapsto z+2$ and $z \mapsto z/(2z+1)$, respectively; the corresponding
 action on the upper half plane is free as well.

The quotient $X(2)=\Gamma(2)\backslash\calH^*$ is again a Riemann surface of
genus $0$; the action of $\Gamma(2)$ on $\PP^1(\Q)$ 
has three orbits, with representatives $0,1,\infty \in
\PP^1 (\Q)$.  We obtain another uniformizing map $\lambda:X(2)
\xrightarrow{\sim} \PP^1(\C)$ with expansion 
\begin{align*}
  \lambda(z)=16q^{1/2} - 128q + 704q^{3/2} - 3072q^2 + 11488q^{5/2} - 38400q^3 +
  \dots.
\end{align*}
As a uniformizer for a congruence subgroup of $\PSL_2(\Z)$, the function
$\lambda (z)$ has a modular interpretation: there is a family of elliptic curves
over $X(2)$ equipped with extra structure.  Specifically, given $\lambda \in
\PP^1(\C) \setminus \{0,1,\infty\}$, the corresponding elliptic curve with extra
structure is given by the \defi{Legendre curve}
\begin{align*}
  E : y^2 = x(x-1)(x-\lambda)  ,
\end{align*}
equipped with the isomorphism $(\Z / 2 \Z)^2 \xrightarrow{\sim} E[2]$ determined
by sending the standard generators to the $2$-torsion points $(0,0)$ and
$(1,0)$.

There is a forgetful map that forgets this additional torsion structure on a
Legendre curve and remembers only isomorphism class; on the algebraic level,
this corresponds to an expression of $j$ in terms of $\lambda$, which is given
by
\begin{equation} \label{eqn:jlambda}
  j(\lambda) = 256\frac{(\lambda^2-\lambda+1)^3}{\lambda^2(\lambda-1)^2};
\end{equation}
indeed, the map $X(2) \to X(1)$ given by $j/1728$ is a Galois \Belyi\ map of
degree $6$ with monodromy group $S_3$, given explicitly by \eqref{eqn:jlambda}.
This map is the Galois closure of the map computed in Example \ref{exm:firstex}.

The cusp $\infty$ plays a special role in the theory of modular forms, and
marking it in our correspondence will allow a suitable rigidification. With this
modification, the correspondence \eqref{eq:freecorresp} becomes a bijection
\begin{equation}\label{eq:G2corresp}
  \begin{gathered} 
    \left\{
      \begin{array}{c}
        \text{transitive permutation triples $\sigma \in S_d$}
        \\
        \text{with a marked cycle of $\sigma_{\infty}$}
      \end{array}
    \right\} /\!\sim
    \\
    \updownarrow \text{\small{1:1}}
        \\
    \left\{ \text{subgroups of $F_2 \cong \Gamma(2)$ of index $d$} \right\}
    /\!\sim
  \end{gathered}
\end{equation}
with equivalence relations as follows: given $\Gamma , \Gamma' \leq \Gamma(2)$,
we have $\Gamma \sim \Gamma'$ if and only if $g \Gamma g^{-1} = \Gamma'$ for
$g$ an element of the subgroups of translations generated by $z \mapsto z + 2$;
and two triples $\sigma,\sigma' \in S_d^3$ with marked cycles $c,c'$ in
$\sigma_{\infty},\sigma'_{\infty}$ are equivalent if and only if they are
simultaneously conjugate by an element $\tau$ with $\tau c \tau^{-1} = c'$.

It is a marvelous consequence of either of the bijections \eqref{eq:freecorresp}
and \eqref{eq:G2corresp}, combined with \Belyi{}'s theorem, that any curve $X$
defined over a number field is uniformized by a subgroup $\Gamma \leq \Gamma(2)
< \PSL_2(\Z)$, so that there is a uniformizing map $\Gamma \backslash \calH^*
\xrightarrow{\sim} X(\C)$.  This is the meaning of Grothendieck's comment: the
rigidification here corresponds to the subgroup $\Gamma$.  In general, the group
$\Gamma$ is \defi{noncongruence}, meaning that it does not contain a subgroup
$\Gamma(N)$, so membership in the group cannot be determined by congruences on
the coordinate entries of the matrices.  This perspective of modular forms is
taken by Atkin--Swinnerton-Dyer \cite{ASD} and Birch \cite[Theorem 1]{Birch} in
their exposition of this subject: they discuss the relationship between modular
forms, the Atkin--Swinnerton-Dyer congruences for noncongruence modular forms,
and Galois representations in the context of \Belyi\ maps.  For more on the
arithmetic aspects of this subject, we refer to the survey by Li--Long--Yang
\cite{LLY} and the references therein.  

The description \eqref{eq:G2corresp} means that one can work quite explicitly
with the Riemann surface associated to a permutation triple.  Given a triple
$\sigma$, the uniformizing group $\Gamma$ is given as the stabilizer of $1$ in
the permutation representation $\Gamma(2) \to S_d$ given by $x,y,z \mapsto
\sigma_0,\sigma_1,\sigma_\infty$ as in \eqref{eq:G2corresp}.  A fundamental
domain for $\Gamma$ is given by \defi{Farey symbols} \cite{KurthLong}, including
a reduction algorithm to this domain and a presentation for the group $\Gamma$
together with a solution to the word problem in $\Gamma$.  These algorithms have
been implemented in the computer algebra systems \texttt{Sage} \cite{Sage} (in a
package for \emph{arithmetic subgroups defined by permutations}, by Kurth,
Loeffler, and Monien) and \texttt{Magma} \cite{Magma} (by Verrill).

Once the group $\Gamma$ has been computed, and the curve $X=\Gamma \backslash
\calH^*$ is thereby described, the \Belyi\ map is then simply given by the
function
\begin{align*}
  \lambda: X \to X(2) \cong \PP^1 ,
\end{align*}
so one immediately obtains an analytic description of \Belyi\ map.  In order to
obtain explicit equations, one needs meromorphic functions on $X$, which is to
say, meromorphic functions on $\calH$ that are invariant under $\Gamma$.  

We are led to the following definition.  Let 
$\Gamma \leq \PSL_2(\Z)$ be a subgroup of finite index.  
A \defi{modular form} for $\Gamma \leq \PSL_2(\Z)$ 
of \defi{weight} $k \in 2\Z$ is a holomorphic function $f: \calH \to \C$ such that 
\begin{equation} \label{eqn:automorphy}
  f(\gamma z)=(cz+d)^k f(z) \quad \text{ for all $\gamma=\pm \begin{pmatrix} a &
    b \\ c & d \end{pmatrix} \in \Gamma$} 
\end{equation}
and such that the limit $\lim_{z \to c} f(z)=f(c)$ exists for all \defi{cusps}
$c \in \Q \cup \{\infty\}=\PP^1(\Q)$ (with the further technical condition that
as $z \to \infty$, we take only those paths that remain in a bounded vertical
strip).  A \defi{cusp form} is a modular form where $f(c)=0$ for each cusp $c$.
The space $S_k(\Gamma)$ of cusp forms for $\Gamma$ of weight $k$ is a
finite-dimensional $\C$-vector space.  If $\Gamma$ is torsion-free or $k=2$,
then there is an isomorphism
\begin{equation} \label{eqn:SKOmega}
  \begin{aligned}
    S_k(\Gamma) &\xrightarrow{\sim} \Omega^{k/2}(X) \\
    f(z) &\mapsto f(z)\,(dz)^{\otimes k/2}
  \end{aligned}
\end{equation}
where $\Omega^{k/2}(X)$ is the space of holomorphic differential $(k/2)$-forms
on $X$.  In any case, evaluation on a basis for $S_k(\Gamma)$ defines a
holomorphic map $\varphi:X \to \PP^{r-1}$ where $r=\dim_\C S_k(\Gamma)$,
whenever $r \geq 1$.  Classical theory of curves yields a complete description
of the map $\varphi$; for example, for generic $X$ of genus $g \geq 3$, taking
$k=2$ (i.e., a basis of holomorphic $1$-forms) gives a \defi{canonical
embedding} of $X$ as an algebraic curve of degree $2g-2$ in $\PP^{g-1}$, by the
theorem of Max Noether.

Selander--Str\"ombergsson \cite{SelanderStrombergsson} use this analytic method
of modular forms to compute \Belyi\ maps; this idea was already present in the
original work of Atkin--Swinnerton-Dyer \cite{ASD} and was developed further by
Hejhal \cite{Hejhal} in the context of Maass forms.   Starting with the analytic
description of a subgroup $\Gamma \leq \Gamma(2)$, they compute a hyperelliptic
model of a curve of genus $2$ from the knowledge of the space $S_2(\Gamma)$ of
holomorphic cusp forms of weight 2 for $\Gamma$.  These cusp forms are
approximated to a given precision by truncated $q$-expansions 
\begin{equation} \label{eqn:fzanqn}
  f(z)=\sum_{n=0}^{N} a_n q^n, 
\end{equation}
one for each equivalence class of cusp $c$ and corresponding local parameter $q$
under the action of $\Gamma$.  These expansions \eqref{eqn:fzanqn} have
undetermined coefficients $a_n \in \C$, and the equation \eqref{eqn:automorphy}
implies an approximate \emph{linear} condition on these coefficients for any
pair of $\Gamma$-equivalent points $z,z'$.  These linear equations can then be
solved using the methods of numerical linear algebra.  This seems to work well
in practice, and once complex approximations for the cusp forms are known, the
approximate algebraic equations that they satisfy can be computed, so that after
a further Newton iteration and then lattice reduction one obtains an exact
solution.  Atkin--Swinnerton-Dyer say of this method \cite[p. 8]{ASD}:
\begin{quote}
  From the viewpoint of numerical analysis, these equations are of course very
  ill-conditioned.  The power series converge so rapidly that one must be
  careful not to take too many terms, and the equality conditions at adjacent
  points in a subdivision of the sides are nearly equivalent. However, by
  judicious choice of the number of terms in the power series and the number of
  subdivision points, for which we can give no universal prescription, we have
  been able to determine the first 8 or so coefficients [...]\ with 7
  significant figures in many cases.
\end{quote}

\begin{ques}
  Does this method give rise to an \emph{algorithm} to compute \Belyi\ maps?  In
  particular, is there an explicit estimate on the numerical stability of this
  method?
\end{ques}

For \Belyi\ maps such that the corresponding subgroup $\Gamma$ is congruence,
methods of \defi{modular symbols} \cite{Cremona,Stein} can be used to determine
the $q$-expansions of modular forms using exact methods.  The Galois groups of
congruence covers are all subgroups of $\PGL_2 (\Z / N \Z)$ for some integer
$N$, though conversely not all such covers arise in this way; as we will see in
the next subsection, since $\PSL_2 (\Z)$ has elliptic points of order $2$ and
$3$, a compatibility on the orders of the ramification types is required.
Indeed, ``most'' subgroups of finite index in $\PSL_2(\Z)$ (in a precise sense)
are noncongruence \cite{JonesStAndrews}. 

\begin{exm} \label{exm:noncong}
  To give a simple example, we consider one of the two (conjugacy classes of)
  noncongruence subgroups of index $7$ of $\PSL_2 (\Z)$, the smallest possible
  index for a noncongruence subgroup by Wohlfarht \cite{Wohlfarht}.  The cusp
  widths of this subgroup are $1$ and $6$.  The information on the cusps tells
  us that the ramification type of the \Belyi\ map above $\infty$ is given by
  $(6,1)$, whereas the indices above $0$ (resp.\ $1$) have to divide $3$ (resp.\
  $2$).  This forces the genus of the dessin to equal $0$, with ramification
  triple $(6^1 1^1,3^2 1^1, 2^3 1^1)$.

  There are exactly two transitive covers with this ramification type, both with
  passport $(0,G,(2^3 1^1,3^2 1^1,6^1 1^1))$. Here the monodromy group $G$ is
  the Frobenius group of order $42$; the two covers correspond
  to two choices of conjugacy classes of order $6$ in $G$.
  For one such choice, we obtain the following unique solution up to conjugacy:
  \begin{align*}
    \sigma_0=(1\ 2)(3\ 4)(6\ 7), \quad \sigma_1=(1\ 2\ 3)(4\ 5\ 6), \quad
    \sigma_{\infty}=(1\ 4\ 7\ 6\ 5\ 3).
  \end{align*}

  A fundamental domain for the action of $\Gamma=\Gamma_7$ is as follows.

  \begin{equation} \label{fig:gamma7} \notag
    \includegraphics[scale=1.2]{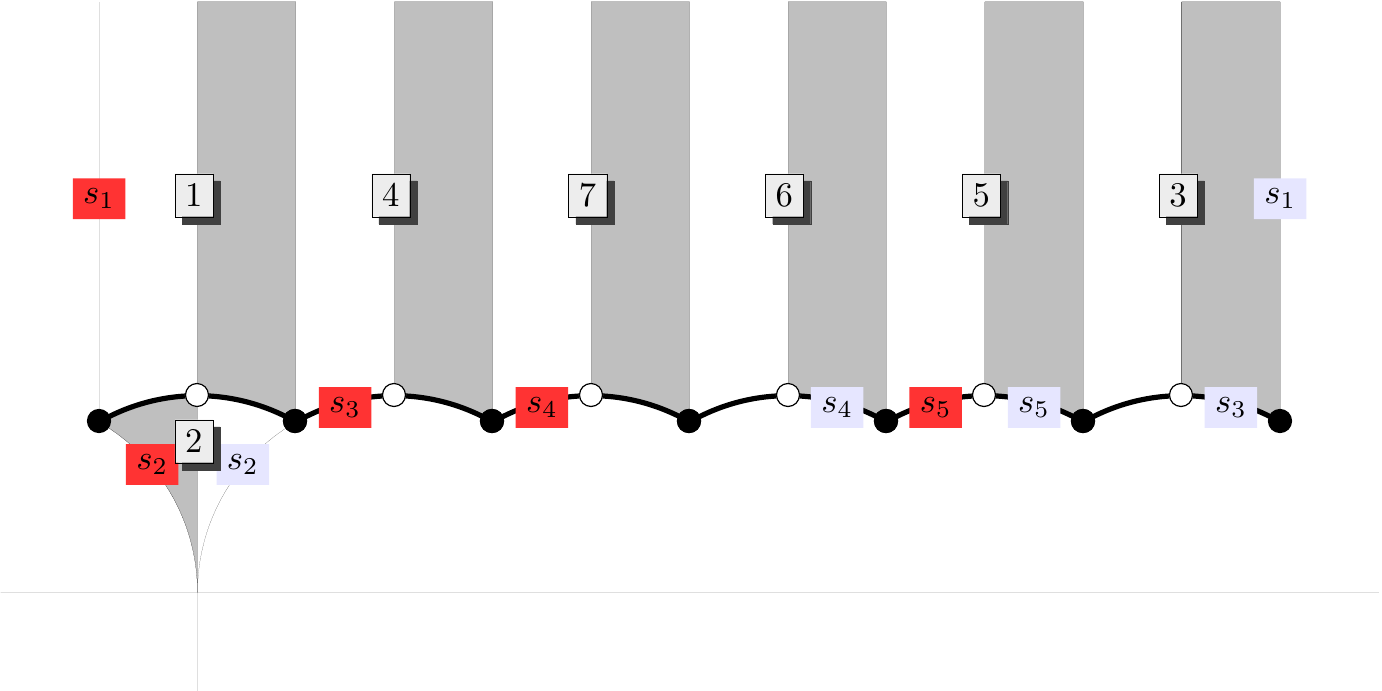}
  \end{equation}

 \begin{center}
    \begin{tabular}{ll}
      \toprule
      Label & Coset Representative\\
      \midrule
      $1$ & $\begin{pmatrix} 1 & 0 \\ 0 & 1 \end{pmatrix}$ \\
      $2$ & $\begin{pmatrix} 0 & -1 \\ 1 & 0 \end{pmatrix}$ \\
      $3$ & $\begin{pmatrix} 1 & 5 \\ 0 & 1 \end{pmatrix}$ \\
      $4$ & $\begin{pmatrix} 1 & 1 \\ 0 & 1 \end{pmatrix}$ \\
      $5$ & $\begin{pmatrix} 1 & 4 \\ 0 & 1 \end{pmatrix}$ \\
      $6$ & $\begin{pmatrix} 1 & 3 \\ 0 & 1 \end{pmatrix}$ \\
      $7$ & $\begin{pmatrix} 1 & 2 \\ 0 & 1 \end{pmatrix}$ \\
      \bottomrule
      \end{tabular}
      \hfill
      \begin{tabular}{ll}
      \toprule
      Label & Side Pairing Element\\
      \midrule
      $s_{1}$ & $\begin{pmatrix} 1 & 6 \\ 0 & 1 \end{pmatrix}$ \\
      $s_{2}$ & $\begin{pmatrix} 1 & 0 \\ 1 & 1 \end{pmatrix}$ \\
      $s_{3}$ & $\begin{pmatrix} 5 & -6 \\ 1 & -1 \end{pmatrix}$ \\
      $s_{4}$ & $\begin{pmatrix} 3 & -7 \\ 1 & -2 \end{pmatrix}$ \\
      $s_{5}$ & $\begin{pmatrix} 4 & -17 \\ 1 & -4 \end{pmatrix}$ \\
      \bottomrule
    \end{tabular}
  \end{center}

  \begin{center}
    \textbf{Figure \ref{fig:gamma7}}: A fundamental domain and side pairing
    for $\Gamma_7 \leq \Gamma(1)$ of index $7$ \\
    \vspace{2ex}
  \end{center}
  \addtocounter{equation}{1}

  We put the cusp of $\Gamma(1)$ at $t=\infty$ and the elliptic point of order
  $3$ (resp.\ $2$) at $t=0$ (resp.z $t=1$).  After this normalization, the $q$-expansion for the
  Hauptmodul $t$ for $\Gamma$ is given by
  \begin{align*}
    t(q)=\frac{1}{\zeta} + 0 + \frac{9+\sqrt{-3}}{2^1 3^4} \zeta +
    \frac{-3-5\sqrt{-3}}{2^2 3^5} \zeta^2 + \frac{1-3\sqrt{-3}}{2^1
    3^7} \zeta^3 + \dots
  \end{align*}
  where $\zeta=\eta q^{1/6}$ and
  \begin{align*}
    \eta^6 = \frac{3^{10}}{7^7}(-1494+3526\sqrt{-3}).
  \end{align*}
  From this, we compute using linear algebra the algebraic relationship between
  $t(q)$ and $j(q)$, expressing $j(q)$ as a rational function in $t(q)$ of
  degree $7$:
  \begin{align*}
    j = -\frac{2^6(1+\sqrt{-3})}{(5-\sqrt{-3})^7}\frac{(54\sqrt{-3}t^2 +
      18\sqrt{-3}t+(5-3\sqrt{-3}))^3(6\sqrt{-3}t -
      (1+3\sqrt{-3}))}{(6\sqrt{-3}t-(1+3\sqrt{-3}))}.
  \end{align*}
  We will compute this example again using $p$-adic methods in the next section 
  (Example \ref{ex:noncong-padic}).
\end{exm}

\subsection*{Modular forms on subgroups of triangle groups}

There is related method that works with a \emph{cocompact} discrete group
$\Gamma \leq \PSL_2(\R)$, reflecting different features of \Belyi\ maps.
Instead of taking the free group on two generators, corresponding to the
fundamental group of $\PP^1 \setminus \{0,1,\infty\}$, we instead consider
orbifold covers arising from triangle groups, a subject of classical interest
(see e.g.\ Magnus \cite{Magnus}).  For an introduction to triangle groups,
including their relationship to \Belyi\ maps and dessins, see the surveys of
Wolfart \cite{WolfartDessins,WolfartObvious}.

Let $a,b,c \in \Z_{\geq 2} \cup \{\infty\}$.  We define the \defi{triangle
group} 
\begin{align*}
  \Delta(a,b,c) = \la \delta_0,\delta_1,\delta_{\infty} \mid
  \delta_0^a=\delta_1^b=\delta_{\infty}^c=\delta_0\delta_1\delta_{\infty}=1 \ra
\end{align*}
where infinite exponents $a,b,c$ are ignored in the relations. Let
$\chi(a,b,c)=1/a+1/b+1/c-1\in \Q$.  For example, we have $\Delta(2,3,\infty)
\cong \PSL_2(\Z)$ and $\Delta(\infty,\infty,\infty) \cong F_2 \cong \Gamma(2)$,
so this construction generalizes the previous section.  The triangle group
$\Delta(a,b,c)$ is the index $2$ orientation-preserving subgroup of the group
generated by the reflections in the sides of a triangle $T(a,b,c)$ with angles
$\pi/a,\pi/b,\pi/c$ drawn in the geometry $H$, where $H=\PP^1,\C,\calH$
according as $\chi(a,b,c)$ is positive, zero, or negative.

Associated to a transitive permutation triple $\sigma$ from $S_d$ is a
homomorphism
\begin{align*} 
  \Delta(a,b,c) &\to S_d \\
  \delta_0,\delta_1,\delta_{\infty} &\mapsto \sigma_0,\sigma_1,\sigma_\infty 
\end{align*}
where $a,b,c \in \Z_{\geq 2}$ are the orders of
$\sigma_0,\sigma_1,\sigma_\infty$, respectively. (Here we have no index
$\infty$, so $\Delta(a,b,c)$ is cocompact, which is where this method diverges
from that using classical modular forms.) The stabilizer of a point $\Gamma \leq
\Delta(a,b,c)$ has index $d$, and the above homomorphism is recovered by the
action of $\Delta$ on the cosets of $\Gamma$.  The quotient map 
\begin{align*}
  \varphi: X = \Gamma \backslash H \to \Delta \backslash H
\end{align*}
then realizes the \Belyi\ map with monodromy $\sigma$, so from this description
we have a way of constructing the \Belyi\ map associated to $\sigma$. More
precisely, as in \eqref{eq:freecorresp}, the bijection \eqref{eq:corresp}
generalizes to
\begin{equation}\label{eq:tricorresp}
  \begin{gathered} 
    \left\{
      \begin{array}{c}
        \text{permutation triples $\sigma = (\sigma_0, \sigma_1,
          \sigma_{\infty}) \in S_d^3$}
        \\
        \text{such that $a,b,c$ are multiples of the orders of
          $\sigma_0,\sigma_1,\sigma_\infty$}
      \end{array}
    \right\} /\!\sim
    \\
    \stackrel{1:1}{\longleftrightarrow}
    \\
    \left\{ \text{subgroups of $\Delta(a,b,c)$ of index $n$} \right\} /\!\sim  ,
  \end{gathered}
\end{equation}
where the equivalences are as usual: conjugacy in the group $\Delta (a,b,c)$ and
simultaneous conjugacy of triples $(\sigma_0 , \sigma_1 , \sigma_{\infty})$.
(In particular, these triples are not marked, as by contrast they are in
\eqref{eq:G2corresp}, though certainly our construction could be modified in
this way if so desired.)

Explicitly, one obtains the Riemann surfaces corresponding to a subgroup $\Gamma
<\Delta(a,b,c)$ under the bijection \eqref{eq:tricorresp} by 
gluing together triangles $T(a,b,c)$ and making
identifications. This gives a conformally correct way to draw dessins and a
method for computing the covers themselves numerically.  

This method has been developed in recent work of Klug--Musty--Schiavone--Voight
\cite{KMSV}.  Algorithms are provided for working with the corresponding
triangle group $\Delta$, determining explicitly the associated finite index
subgroup $\Gamma$, and then drawing the dessin on $H$ together with the gluing
relations that define the quotient $X=\Gamma \backslash H$.  From this explicit
description of the Riemann surface (or more precisely, Riemann $2$-orbifold) $X$
one obtains equations for the \Belyi\ map $f$ numerically.  The main algorithmic
tool for this purpose is a generalization of Hejhal's method replacing
$q$-expansions with power series expansions, due to Voight--Willis
\cite{VoightWillis}.  This method works quite well in practice; as an
application, a \Belyi\ map of degree $50$ of genus $0$ regularly realizing the
group $\PSU_3(\F_5)$ over $\Q(\sqrt{-7})$ is computed.  

\begin{exm}
  Consider the permutation triple $\sigma = (\sigma_0 , \sigma_1 ,
  \sigma_{\infty})$, where
  \begin{align*}
    \sigma_0 &= (1\ 7\ 4\ 2\ 8\ 5\ 9\ 6\ 3) \\
    \sigma_1 &=(1\ 4\ 6\ 2\ 5\ 7\ 9\ 3\ 8) \\
    \sigma_\infty &= (1\ 9\ 2)(3\ 4\ 5)(6\ 7\ 8) .
  \end{align*}
  Then $\sigma_0\sigma_1\sigma_\infty=1$ and these permutations generate a
  transitive subgroup \[ G \cong \Z/3\Z ~\wr~\Z/3\Z \leq S_{9} \] of order $81$ and
  give rise to a \Belyi\ map with passport $(0,G,(9^1,9^1,3^3))$.  The corresponding
  group $\Gamma \leq \Delta(9,9,3)=\Delta$ of index $9$ arising from
  \eqref{eq:tricorresp} has signature $(3;-)$, i.e., the quotient $\Gamma
  \backslash \calH$ is a (compact) Riemann surface of genus $3$.  The map
  $X(\Gamma) = \Gamma \backslash \calH \to X(\Delta) = \Delta \backslash \calH
  \cong \PP^1$ gives a \Belyi\ map of degree $9$, which we now compute.

  First, we compute a \defi{coset graph}, the quotient of the Cayley graph for
  $\Delta$ on the generators $\delta_0^{\pm}, \delta_1^{\pm}$ by $\Gamma$ with
  vertices labelled with coset representatives $\Gamma \alpha_i$ for $\Gamma
  \backslash \Delta$.  Given a choice of fundamental domain $D_\Delta$ for
  $\Delta$ (a fundamental triangle and its mirror, as above), such a coset graph
  yields a fundamental domain $D_\Gamma = \bigcup_{i=1}^n \alpha_i D_\Delta$
  equipped with a \defi{side pairing}, indicating how the resulting Riemann
  orbifold is to be glued.  We consider this setup in the unit disc $\calD$,
  identifying $\calH$ conformally with $\calD$ taking a vertex to the center
  $w=0$; the result is Figure \ref{fig:trianglefd}.  We obtain in this way a
  reduction algorithm that takes a point in $z \in \calH$ (or $\calD$) and
  produces a representative $z' \in D_\Gamma$ and $\gamma \in \Gamma$ such that
  $z'=\gamma z$.  

  We consider the space $S_2(\Gamma)$ of cusp forms of weight $2$ for $\Gamma$,
  defined as in \eqref{eqn:automorphy} (but note that since no cusps are present
  we can omit the corresponding extra conditions). As in \eqref{eqn:SKOmega}, we
  have an isomorphism $S_2(\Gamma) \cong \Omega^1(X)$ of $\C$-vector spaces with
  the space of holomorphic $1$-forms on $X$.  Since $X$ has genus $3$, we have
  $\dim_\C S_2(\Gamma)=3$.  We compute a basis of forms by considering power
  series expansions
  \begin{align*}
    f(w) = (1-w)^2 \sum_{n=0}^{\infty} b_n w^n
  \end{align*}
  for $f \in S_2(\Gamma)$ around $w=0$ in the unit disc $\calD$.  (The presence
  of the factor $(1-w)^2$ makes for nicer expansions, as below.)  We compute
  with precision $\epsilon=10^{-30}$, and so $f(w) \approx (1-w)^2
  \sum_{n=0}^{N} b_n w^n$ with $N=815$.  We use the Cauchy integral formula to
  isolate each coefficient $b_n$, integrating around a circle of radius
  $\rho=0.918711$ encircling the fundamental domain. This integral is approximated by summing the evaluations at $O(N)$ points on this circle, which can be explictly represented by elements in the fundamental domain $D_\Gamma$ after using the reduction algorithm.

  \begin{equation} \label{fig:trianglefd} \notag
    \includegraphics[scale=0.88]{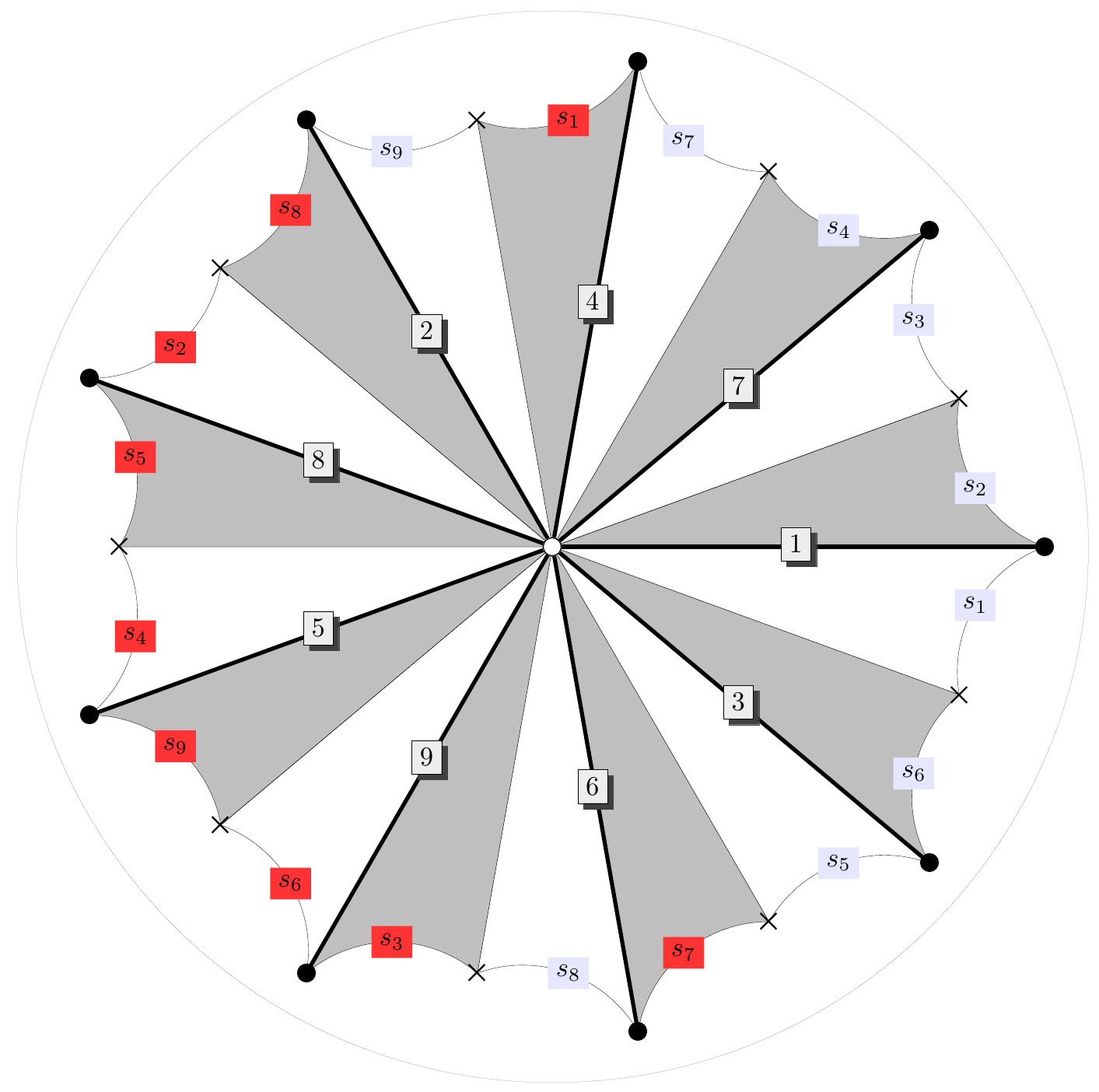}
  \end{equation}

  \begin{center}
    \begin{tabular}{ll}
      \toprule
      Label & Coset Representative\\
      \midrule
      $1$ & 1 \\
      $2$ & $\delta_0^{3}$ \\
      $3$ & $\delta_0^{-1}$ \\
      $4$ & $\delta_0^{2}$ \\
      $5$ & $\delta_0^{-4}$ \\
      $6$ & $\delta_0^{-2}$ \\
      $7$ & $\delta_0^{}$ \\
      $8$ & $\delta_0^{4}$ \\
      $9$ & $\delta_0^{-3}$ \\
      \bottomrule
      \end{tabular}
      \hfill
      \begin{tabular}{ll}
      \toprule
      Label & Side Pairing Element\\
      \midrule
      $s_{1}$ & $\delta_1^{}\delta_0^{-2}$ \\
      $s_{2}$ & $\delta_1^{-1}\delta_0^{-4}$ \\
      $s_{3}$ & $\delta_0^{}\delta_1^{}\delta_0^{3}$ \\
      $s_{4}$ & $\delta_0^{}\delta_1^{-1}\delta_0^{4}$ \\
      $s_{5}$ & $\delta_0^{-1}\delta_1^{}\delta_0^{-4}$ \\
      $s_{6}$ & $\delta_0^{-1}\delta_1^{-1}\delta_0^{3}$ \\
      $s_{7}$ & $\delta_0^{2}\delta_1^{}\delta_0^{2}$ \\
      $s_{8}$ & $\delta_0^{-2}\delta_1^{}\delta_0^{-3}$ \\
      $s_{9}$ & $\delta_0^{3}\delta_1^{}\delta_0^{4}$ \\
      \bottomrule
    \end{tabular}
  \end{center}

  \begin{center}
    \textbf{Figure \ref{fig:trianglefd}}: A fundamental domain and side pairing
    for $\Gamma \leq \Delta(9,9,3)$ of index $9$ \\
    \vspace{2ex}
  \end{center}
  \addtocounter{equation}{1}

  We find the \defi{echelonized} basis
  \begin{align*}
    x(w) = (1-w)^2 &\left( 1 - \frac{40}{6!}(\Theta w)^6 +
      \frac{3080}{9!}(\Theta w)^9 - \frac{1848000}{12!}(\Theta w)^{12} +
      O(w^{15})\right) \\
    y(w) = (1-w)^2 &\left( (\Theta w) + \frac{4}{4!}(\Theta w)^4 +
      \frac{280}{7!}(\Theta w)^{7} - \frac{19880}{10!}(\Theta w)^{10} +
      O(w^{13})\right) \\
    z(w) = (1-w)^2 &\left( (\Theta w)^3 - \frac{120}{6!}(\Theta w)^6 -
      \frac{10080}{9!}(\Theta w)^{9} - \frac{2698080}{12!}(\Theta w)^{12} +
      O(w^{15})\right)
  \end{align*}
  where $\Theta = 1.73179\ldots + 0.6303208\ldots \sqrt{-1}$.  The algebraicity
  and near integrality of these coefficients are conjectural \cite{KMSV}, so
  this expansion is only numerically correct, to the computed precision.

  We now compute the image of the canonical map 
  \begin{align*}
  X(\Gamma) = \Gamma \backslash \calH &\to \PP^2 \\
  w &\mapsto (x(w):y(w):z(w)); 
  \end{align*}
  we find a unique quartic relation
  \begin{align*}
    216 x^3 z - 216 xy^3 + 36 xz^3 + 144y^3 z - 7z^4 = 0
  \end{align*}
  so at least numerically the curve $X$ is nonhyperelliptic.  Evaluating these
  power series at the ramification points, we find that the unique point above
  $f=0$ is $(1: 0: 0)$, the point above $f=1$ is $(1/6: 0: 1)$, and the three
  points above $f=\infty$ are $(0:1:0)$ and $((-1\pm 3\sqrt{-3})/12 : 0 : 1)$.  

  The uniformizing map $f:X(\Gamma) \to X(\Delta) \cong \PP^1$ is given by the
  reversion of an explicit ratio of hypergeometric functions:
  \begin{align*}
    f(w)= -\frac{1}{8}(\Theta w)^9 - \frac{11}{1280}(\Theta w)^{18} -
      \frac{29543}{66150400}(\Theta w)^{27} + O(w^{36}).
  \end{align*}
  Using linear algebra, we find the expression for $f$ in terms of $x,y,z$: 
  \begin{align*}
    f(w) = \frac{-27z^3}{216x^3 - 108x^2z + 18xz^2 - 28z^3}.
  \end{align*}
  Having performed this numerical calculation, we then verify on the curve
  $X(\Gamma)$ that this rational function defines a three-point cover with the
  above ramification points, as in Section \ref{sec:veri}.
\end{exm}

An important feature of methods using modular forms is that it allows a much
more direct algebraic approach to determining the algebraic structure on the
target Riemann surface. There are no ``parasitic'' solutions to discard, just as
when using the more advanced analytic method of Section \ref{sec:coan}.
Moreover, the equation for the source surface are much easier to find than with
the analytic method, where one typically needs to compute period matrices to
high precision.

\begin{ques}
  What are the advantages of the \emph{noncocompact} ($q$-expansions)
  and \emph{cocompact} (power series expansions) approaches relative to one another?  How far (degree,
  genus) can these methods be pushed?  Can either of these methods be made rigorous?
\end{ques}

\section{$p$-adic methods}\label{sec:padic} 

As an alternative to complex analytic methods, we can use $p$-adic methods to
find a solution; in this section we survey this method, and give a rather
elaborate example of how this works in practice. It is simply the $p$-adic
version of the complex analytic method, with the big distinction that finding a
suitable approximation and then Hensel lifting can be much easier; usually
finding a solution over a finite field suffices to guarantee convergence of
Newton approximation.

\subsection*{Basic idea}

The $p$-adic method begins by finding a solution in a finite field of small
cardinality, typically by exhaustive methods, and then lifts this solution using
$p$-adic Newton iteration.  Again, lattice methods can be then employed to
recognize the solution over $\Qbar$.  \emph{Turning the `$p$-adic crank'}, as it
is called, has been a popular method, rediscovered many times and employed in a
number of contexts.  Malle \cite{MalleM22} used this method to compute
polynomials with Galois groups $M_{22}$, $\Aut(M_{22})$, and $\PSL_3(\F_4):2$
over $\Q$.  Elkies \cite{Elkies237} computed a degree $28$ cover $f:X \to \PP^1$
with group $G=\PSL_2(\F_{27})$ via its action on $\PP^1(\F_{27})$ modulo $29$,
and other work of Elkies \cite{ElkiesM23}, Watkins \cite{Watkins} and
Elkies--Watkins \cite{ElkiesWatkins} have also successfully used $p$-adic
methods to compute \Belyi\ maps.  Elkin--Siksek \cite{ElkinSiksek} used this
method and tabulated \Belyi\ maps of small degree. Van Hoeij--Vidunas
\cite{vHV2} used this approach to compute a list of examples whose branching is
nearly regular, before extending the direct method \cite{vHV} as explained in
Section \ref{sec:groeb}. More recently, Bartholdi--Buff--von Bothmer--Kr\"oker
\cite{Bartholdi} computed a \Belyi\ map in genus $0$ that is of degree $13$ and
which arises in a problem of Cui in dynamical systems; they give a relatively
complete description of each of the steps involved.

A foundational result by Beckmann indicates which primes are primes of good
reduction for the \Belyi\ map; which primes, therefore, can be used in the
procedure above.

\begin{thm}[Beckmann \cite{Beckmann}]\label{thm:Beckmann}
  Let $f:X \to \PP^1$ be a \Belyi\ map and let $G$ be the monodromy group of $f$.  
  Suppose that $p \nmid \# G$.  Then there exists a number field $L$ such that
  $p$ is unramified in $L$ and $f$ is defined over $L$ with good reduction at 
  all primes $\frakp$ of $L$ lying over $p$.
\end{thm}

\begin{rmk}
  In fact, Beckmann proves as a consequence that under the hypotheses of the
  theorem, the prime $p$ is unramified in the field of moduli $K$ of $f$.  (For
  the definition of the \defi{field of moduli}, see Section \ref{sec:fomfod}.)
\end{rmk}

If one works with a pointed cover instead, then the statement of Beckmann's
theorem is simpler \cite[Theorem 3]{Birch}.  In the notation of this theorem, if
$p$ divides the order of one of the permutations $\sigma$ then $f$ has bad
reduction at $\frakp$ \cite[Theorem 4]{Birch}. But for those $p$ that divide $\#
G$ but not any of the ramification indices, it is much harder to find methods
(beyond explicit calculation) to decide whether or not a model of $f$ with good
reduction over $\frakp$ exists.  Important work in this direction is due to
Raynaud \cite{Raynaud} and Obus \cite{Obus}.  

\begin{ques}
  Can one perform a similar lifting procedure by determining solutions modulo
  primes where $f$ has bad reduction?
\end{ques}

As the matrix of derivatives of the equations used is almost always of full rank
(see Section \ref{sec:groeb}), the most time-consuming part is usually the
search for a solution over a finite field. In order for this method to be
efficient, one must do better than simply running over the potential solutions
over $\F_q$.  Bartholdi--Buff--van Bothmer--Kr\"oker describe \cite[Algorithm
4.7]{Bartholdi} a more careful method for genus $0$, working directly with
univariate polynomials (and rational functions) with coefficients in $\F_q$. In
the example below, we show an approach that is similar in spirit to theirs and
that works for hyperelliptic curves as well.

When the field of definition is ``generic'' in some sense, then there is often a
split prime of small norm, so this method is often efficient in practice. The
following question still merits closer investigation.

\begin{ques}
  How efficiently can a \Belyi\ map be computed modulo a prime $p$? How far can
  one reduce the dimension of the affine space employed in the enumeration? 
\end{ques}

In particular, can a ``partial projection'' (partial Gr\"obner basis) be
computed efficiently to reduce the number of looping variables?

\begin{exm} \label{ex:noncong-padic}
  We return to the \Belyi\ maps with ramification type $(6^1 1^1, 3^2 1^1, 2^3
  1^1)$ considered in Example \ref{exm:noncong}.  

  Theorem \ref{thm:Beckmann} suggests to reduce modulo $5$ first. We put the
  ramification type $(6,1)$ over $\infty$ and the corresponding points at
  $\infty$ and $0$; we can do this without risking an extension of the field of
  definition since these points are unique. In the same way, we put the type
  $(3^2,1)$ over $0$ and the single point in this fiber at $1$. This defines a
  reasonably small system over $\F_5$ of dimension $7$, which could even be
  checked by enumeration. We get the solutions
  \begin{align*}
    f(t) = \frac{\alpha^{8} (t - 2)^3 (t + \alpha)^3 (t - 1)}{t}
  \end{align*}
  and its conjugate, where $\alpha$ is a root of the Conway polynomial defining
  $\F_{5^2}$ over $\F_5$, i.e., $\alpha^2-\alpha+2=0$. At the prime $13$, we
  get two solutions defined over $\F_{13}$:
  \begin{align*}
    f(t) = \frac{-3 (t^2 + 3t + 8)^3 (t - 1)}{t}, \quad f(t)= \frac{2 (t^2 +
      6)^3 (t - 1)}{t}.
  \end{align*}

  In both cases, the derivative matrices of the equations (with or without ASD)
  are non-singular, so we can lift to the corresponding unramified $p$-adic
  fields. After a few iterations of the second pair of solutions, we get
  the $13$-adic approximations
  \begin{align*}
    f(t) &= (-3 -5 \cdot 13 - 13^2 + \dots )(t - 1)t^{-1}  \\
    &\qquad \cdot (t^2 + (3+8\cdot 13 - 2\cdot 13^2+\dots)t + (8 - 3\cdot 13 -
      6\cdot 13^2 + \dots))^3 \\
    f(t) &= (2 -3 \cdot 13 + 3 \cdot 13^2 + \dots)(t - 1)t^{-1}  \\
    &\qquad \cdot (t^2 + (-4 \cdot 13 + 6 \cdot 13^2 + \dots) t + (6 - 3 \cdot
      13^2 + \dots))^3.
  \end{align*}

  We continue, with quadratically growing accuracy, in order to use LLL in the
  end. This suggests a pair of solutions over $\Q (\sqrt{-3})$ given by
  \begin{align*}
    f(t) = \frac{-1+\sqrt{-3}}{4\sqrt{-3}^3 (\sqrt{-3} + 2)^7} \frac{(162 t^2 +
      18 (-\sqrt{-3} - 6) t + (\sqrt{-3} + 3))^3 (t - 1)}{t}
   \end{align*}
  and its conjugate. One verifies as in Section \ref{sec:veri} that this yields
  a solution over $\Q(\sqrt{-3})$ to the given equations and that they are the
  requested \Belyi\ maps.  Though we stop here, one could further simplify the
  equation even further by suitable scalar multiplications in $t$, or even
  better, the general methods described in Section \ref{sec:veri}.
\end{exm}

\begin{exm}
We now illustrate the complexities involved in employing the above method in an
example. It arose during a study of Galois \Belyi\ maps with monodromy group
$\PSL_2(\F_q)$ or $\PGL_2(\F_q)$, undertaken by Clark--Voight
\cite{ClarkVoight}. 

Consider the passport with uniform ramification orders $3,5,6$ and monodromy
group $G = \PSL_2(\F_{11}) \leq S_{11}$. Here the embedding of $G$ in $S_{11}$
results from its conjugation action on the cosets of its exceptional subgroup $A_5$
(and indeed $\# G /\#A_5 = 660/60 = 11$).

Let $f : E \to \PP^1$ be the degree $11$ \Belyi\ map defined by the above data,
and let $\varphi: X \to \PP^1$ be its Galois closure, with Galois group $G$.  We
anticipate \cite{ClarkVoight} that $\varphi$ with its Galois action is defined
over an at most quadratic extension of $\Q(\sqrt{3},\sqrt{5})$, in which case by
the Galois correspondence the quotient map $f$ will be defined over the same field. We confirm this
by direct computation.

Using the representation of $G$ above, we find that $f$ has 
passport 
\[ (1,\PSL_2(\F_{11}),( 3^3 1^2 , 5^2 1^1 , 6^1 3^1 2^1)); \] 
in accordance with the construction above, the ramification orders are
divisors of $3,5,6$, and $E$ has genus $1$.  

We distinguish the point of ramification degree $6$ above $\infty$ and obtain a
corresponding group law on $E$. We fix two more points by taking the other
points above $\infty$ (with ramification $3$ and $2$, respectively) to be
$(0,1)$ and $(1,y_1)$. We write the equation
\begin{align*}
  y^2 = \pi_3 x^3 + \pi_2 x^2 + (y_1^2-\pi_3-\pi_2-1)x+1 = \pi(x) 
\end{align*}
for the curve $E$. The \Belyi\ function $f$ has the form
\begin{align*}
  f (x,y) = \frac{q(x)+r(x)y}{(x-1)^2 x^3} 
\end{align*}
where $q(x)=q_8 x^8 +\dots+q_0$ and $r(x)=r_6x^6+\dots+r_0$ have degree $8,6$
respectively and the numerator $f\sbnum(x,y)=q(x)+r(x)y$ vanishes to degree $3$
at $(0,-1)$ and $2$ at $(0,-y_1)$.  

By the ramification description above $0$, we must have
\begin{equation}\label{phinum}
  \begin{gathered}
    \begin{aligned}
      \N_{\Qbar (x,y) / \Qbar (x)} ( f\sbnum(x,y) ) & = q(x)^2 - r(x)^2 \pi(x) \\
      & = q_8^2 x^3 (x-1)^2 s(x)^3 t(x)
    \end{aligned}
  \end{gathered}
\end{equation}
where $s(x)=x^3+s_2x^2+s_1x+s_0$ and $t(x) = x^2 + t_1 x + t_0$, and similarly
above $1$ we should have
\begin{equation}\label{phinum1}
  \begin{gathered}
    \begin{aligned}
      \N_{\Qbar (x,y) / \Qbar (x)} ((f(x,y)-1)\sbnum) & = (q(x)-(x-1)^2x^3)^2 -
      r(x)^2 \pi(x) \\
      & = q_8^2 x^3 (x-1)^2 u(x) v(x)
    \end{aligned}
  \end{gathered}
\end{equation}
where $u(x) = x^2 + u_1 x + u_0$ and $v(x) = x + v_0$.

An approach using Gr\"obner basis techniques utterly fails here, given the
number of variables involved.  This calculation is also made more difficult by
the possibility that other \Belyi\ covers will intervene: the Mathieu group
$M_{11} \hookrightarrow S_{11}$ also has a $(3,5,6)$ triple of genus $1$, and it
is a priori conceivable that $S_{11}$ occurs as well.  Discarding these
parasitic solutions is a nontrivial task until one has already computed all of
them along with the correct ones, just as in Section \ref{sec:groeb}.

As explained above, we search for a solution in a finite field $\F_q$, lift such
a solution using Hensel's lemma (if it applies), and then attempt to recognize
the solution $p$-adically as an algebraic number using the LLL lattice reduction
algorithm.  The primes of smallest norm in the field $\Q(\sqrt{3},\sqrt{5})$
that are relatively prime to $\# \PSL_2(\F_{11})$ have norm $q=49,59$, so there
is no hope of simply running over all the $\F_q$-rational values in the affine
space in $y_1 , \pi , q , r , s , t , u , v$, which is $28$-dimensional.

We speed up the search with a few tricks.  Subtracting the two equations
(\ref{phinum})--(\ref{phinum1}), we have
\begin{align*}
  q_8^2 s(x)^3 t(x) - 2q(x) + (x-1)^2 x^3 = r_8^2 u(x)^5v(x). 
\end{align*}
Comparing coefficients on both sides, by degree we see that the coefficients of
$x^9$ and $x^{10}$ of $s(x)^3t(x)$ and $u(x)^5v(x)$ must agree.  So we
precompute a table of the possible polynomials of the form $u(x)^5 v(x)$; there
are $O(q^3)$ such, and we sort them for easy table lookup.  Then, for each of
the possible polynomials of the form $s(x)^3 t(x)$, of which there are
$O(q^5)$, we match the above coefficients. Typically there are few matches.
Then for each $q_8^2 \in \F_q^{\times 2}$, we compute $q(x)$ as 
\begin{align*}
  q(x) = \frac{1}{2} \left(q_8^2 s(x)^3 t(x) - q_8^2 u(x)^5v(x) - (x-1)^2
  x^3\right). 
\end{align*}
From equation (\ref{phinum}) we have
\begin{align*}
  q(x)^2 - q_8^2(x-1)^2x^3 s(x)^3 t(x) = \pi(x)r(x)^2  ,
\end{align*}
so we compute the polynomial on the right and factor it into squarefree parts.
If the corresponding $\pi(x)$ has degree $3$, then we find $r(x)$ as well,
whence also our solution.

Putting this on a cluster at the Vermont Advanced Computing Center (VACC) using
\textsf{Magma} \cite{Magma}, after a few days we have our answer.  We find
several solutions in $\F_{49}$ but only one solution lifts $p$-adically without
additional effort; it turns out the Jacobian of the corresponding system of
equations is not of full rank. After some effort (see also Section
\ref{sec:veri}), we recognize this cover as an $M_{11}$-cover with ramification
$(3,5,6)$, defined over the number field $\Q(\alpha)$ where 
\begin{align*}
  \alpha^7 - \alpha^6 - 8\alpha^5 + 21\alpha^4 + 6\alpha^3 - 90\alpha^2 +
  60\alpha + 60 = 0. 
\end{align*}

We find $62$ solutions in $\F_{59}$. Note that the $M_{11}$-covers above do not
reappear since there is no prime of norm $59$ in $\Q(\alpha)$. Only $8$ of
these solutions yield covers with the correct ramification data; our above
conditions are necessary, but not sufficient, as we have only considered the
$x$-coordinates and not the $y$-coordinates. These $8$ covers lift to a single
Galois orbit of curves defined over the field $\Q(\sqrt{3},\sqrt{5},\sqrt{b})$
where
\begin{align*}
  b=4\sqrt{3} + \frac{11+\sqrt{5}}{2}; 
\end{align*}
with $N(b)=11^2$; more elegantly, the extension of $\Q(\sqrt{3},\sqrt{5})$ is
given by a root $\beta$ of the equation
\begin{align*}
  T^2 - \frac{1+\sqrt{5}}{2}T - (\sqrt{3}+1) = 0. 
\end{align*}

The elliptic curve $E$ has minimal model
\begin{align*}
  &y^2 + ((\textstyle{\frac{1}{2}}(13\sqrt{5} + 33)\sqrt{3} +
    \textstyle{\frac{1}{2}}(25\sqrt{5} + 65))\beta +
    (\textstyle{\frac{1}{2}}(15\sqrt{5} + 37)\sqrt{3} + (12\sqrt{5} + 30)))xy \\
  &\quad + (((8\sqrt{5} + 15)\sqrt{3} + \textstyle{\frac{1}{2}}(31\sqrt{5} +
    59))\beta + (\textstyle{\frac{1}{2}}(13\sqrt{5} + 47)\sqrt{3} +
    \textstyle{\frac{1}{2}}(21\sqrt{5} + 77)))y  \\
  &\quad = x^3 + ((\textstyle{\frac{1}{2}}(5\sqrt{5} + 7)\sqrt{3} +
    \textstyle{\frac{1}{2}}(11\sqrt{5} + 19))\beta +
    (\textstyle{\frac{1}{2}}(3\sqrt{5} + 17)\sqrt{3} + (2\sqrt{5} + 15)))x^2  \\
  &\quad + ((\textstyle{\frac{1}{2}}(20828483\sqrt{5} + 46584927)\sqrt{3} +
    \textstyle{\frac{1}{2}}(36075985\sqrt{5} + 80687449))\beta \\
  &\qquad    +(\textstyle{\frac{1}{2}}(21480319\sqrt{5} + 48017585)\sqrt{3} +
    \textstyle{\frac{1}{2}}(37205009\sqrt{5} + 83168909)))x  \\
  &\quad + (((43904530993\sqrt{5} + 98173054995)\sqrt{3} +
    \textstyle{\frac{1}{2}}(152089756713\sqrt{5} + 340081438345))\beta \\
  &\qquad +((45275857298\sqrt{5} + 101240533364)\sqrt{3} + (78420085205\sqrt{5}
    + 175353747591))).
\end{align*}
The $j$-invariant of $E$ generates the field $\Q(\sqrt{3},\sqrt{5},\beta)$, so
this is its minimal field of definition.  This confirms that $\varphi:X \to
\PP^1$ as a $G$-cover is defined over an at most quadratic extension of
$\Q(\sqrt{3},\sqrt{5})$ contained in the ray class field of conductor
$11\infty$, as predicted by the results of Clark--Voight \cite{ClarkVoight}.  

\end{exm}

\section{Galois \Belyi\ maps}\label{sec:galois}

In this short section we sketch some approaches for calculating Galois \Belyi\
maps, i.e., those \Belyi\ maps $f:X \to \PP^1$ corresponding to Galois
extensions of function fields.  The flavor of these computations is completely
different from those in the other sections, as the representation-theoretic
properties of the Galois group involved are used heavily.  In light of the
Galois correspondence, all \Belyi\ maps are essentially known once the Galois
\Belyi\ maps are known; however, the growth in degree between the degree of the
\Belyi\ map and that of its Galois closure makes it very difficult in general to
make this remark a feasible approach to computing general \Belyi\ maps. We
therefore consider the subject only in itself, and even here we limit ourselves
to the general idea: exploiting representations and finding invariant functions.

The Galois \Belyi\ maps in genus $0$ correspond to the regular solids, and can
be computed using the direct method (see the end of Section
\ref{sec:directcalc}).  The most difficult case, that of the icosahedron, was
calculated first by Klein \cite{Klein}.  The Galois \Belyi\ maps in genus $1$
only occur on curves with CM by either $\Q(\sqrt{-3})$ or $\Q(\sqrt{-1})$, and
can therefore be calculated by using explicit formulas for isogenies; see work
of Singerman--Syddall \cite{SingermanSyddall}.

So it remains to consider the case of genus $\geq 2$, where \Belyi\ maps are
related with hyperbolic triangle groups (see Section \ref{sec:mod}).  In genus
$\geq 2$, Wolfart \cite{WolfartObvious} has shown that Galois \Belyi\ maps can
be identified with quotient maps of curves \defi{with many automorphisms}, that
is, those curves that do not allow nontrivial deformations that leave the
automorphism group intact and whose automorphism group therefore defines a
zero-dimensional subscheme of the moduli space of curves $\calM_g$ of genus $g
\geq 2$. Wolfart \cite{WolfartCM} compares these \Belyi\ maps with the related
phenomenon of \defi{Jacobians of CM type}, which define zero-dimensional
subschemes of the moduli space of principally polarized abelian varieties
$\calA_g$. In particular, the CM factors of the Jacobians of the Galois
\Belyi\ curves are essentially known; they come from Fermat curves \cite[\S
4]{WolfartCM}.

A fundamental technique for proving these theorems is to determine the
representation of the automorphism group on the space of differentials, first
considered by Chevalley and Weil \cite{ChevalleyWeil}; this is
elaborated by Berry--Tretkoff \cite{BT} and Streit \cite{StreitHBC}. Once
this is done, one typically recovers the curve by determining the shape of its
canonical embedding, often an intersection of quadrics.  (When the canonical
embedding is not injective, the situation is even simpler; since the
hyperelliptic involution is central in the automorphism group, this reduce to
the calculations in genus $0$ mentioned above.)  The particular form of the
equations is then determined by being fixed under the action of the automorphism
group, which acts by linear transformations.

\begin{ques}
  Can the representation of the automorphism group $G$ on the space of
  differentials be used to give a rigorous algorithm for the computation of
  $G$-Galois \Belyi\ maps (with a bound on the running time)?
\end{ques}

Put another way, computing a Galois \Belyi\ map amounts to determining
$G$-invariant polynomials of a given degree; in some cases, there is a unique
such polynomial with given degree and number of variables, and so it can be
found without any computation.

\begin{exm}\label{exm:gal1}
  We illustrate the invariant theory involved by giving an example of a
  calculation of a quotient map $X \to X / \Aut(X) \cong \PP^1$ that turns out
  to be a \Belyi\ map; the example was suggested to us by Elkies.

  Consider the genus $9$ curve $X$ defined by the following variant of the Bring
  equations:
  \begin{align*}
    v + w  + x  + y  + z  &= 0  ,  \\
    v^2 + w^2  + x^2  + y^2  + z^2 &= 0  ,  \\
    v^4 + w^4  + x^4  + y^4  + z^4 &= 0  .
  \end{align*}
  This curve is known as \defi{Fricke's octavic curve}, and it was studied
  extensively by Edge \cite{Edge}.  There is an obvious linear action by $S_5$
  on this curve by permutation of coordinates. To find coordinates on the
  quotient $X / S_5$ it therefore suffices to look at the symmetric functions in
  the variables $v,w,x,y,z$. We see that the power sums with exponents $1,2,4$
  vanish on $X$.  Since the ring of invariants function for $S_5$ is generated
  by the power sums of degree at most $5$, this suggests that we cook up a
  function from the power sums $p_3$ and $p_5$ of degree $3$ and $5$.  These
  functions do not have the same (homogeneous) degree; to get a well-defined
  function, we consider their quotient $f = (p_3^5 : p_5^3)$ as a morphism from
  $f:X \to \PP^1$.

  The intersection of the hyperplanes defined by $p_3 = 0$ and $p_5 = 0$ with
  $X$ are finite; indeed, this is obvious since the corresponding functions do
  not vanish indentically on $X$. By B\'ezout's theorem, these zero loci are of
  degrees $24,40$.  But whereas in the former case one indeed obtains
  $24$ distinct geometric points in the intersection, one obtains only $20$
  geometric points in the latter case. This shows that the ramification indices
  over $0$ and $\infty$ of the degree $120$ morphism $f$ are $6$ and $5$.

  This is in fact already enough to conclude that there is only one other branch
  point for $q$. Indeed, the orbifold $X / \Aut (X)$ is uniformized by the upper
  half plane $\calH$ since the genus $9$ curve $X$ is, so $X / \Aut (X)$ is a
  projective line with at least $3$ branch points for the quotient by the action
  of $S_5$. On the other hand, the Riemann--Hurwitz formula shows that adding a
  single minimal contribution of $2$ outside the contributions $5$ and $6$
  already known from $\infty$ and $0$ already makes the genus grow to $9$, so
  additional ramification is impossible.  The additional branch point of $f$ can
  be found by considering the divisor of $d f$ on $X$; this point turns out to
  be $- (15/2)^2$. So the morphism $f : X \to \PP^1$ defined by
  \begin{align*}
    f (v,w,x,y,z) = \frac{-2^2 p_3^5}{15^2 p_5^3} = -
    \left(\frac{2}{15}\right)^2 \frac{(v^3 + w^3 + x^3 + y^3 + z^3)^5}{(v^5 +
    w^5 + x^5 + y^5 + z^5)^3}
  \end{align*}
  realizes the quotient $X \to X / S_5$ as a \Belyi\ map. Moreover, we see that
  the Galois action is defined over $\Q$, since it is given by permuting the
  given coordinate functions on $X$.

  In fact we have an isomorphism $\Aut(X) \cong S_5$ since $\Aut (X)$ cannot be
  bigger than $S_5$; such a proper inclusion would give rise to a Fuchsian group
  properly containing the triangle group $\Delta(2,5,6)$, whereas on the other
  hand this group is maximal (by work of Takeuchi \cite{Takeuchi}, or more
  generally see Singerman \cite{Singerman} or Greenberg \cite[Theorem
  3B]{Greenberg}).

  We therefore have found a Galois cover realizing $S_5$ with ramification
  indices $2,5,6$. It turns out that this is the only such cover up to
  isomorphism. Considering the exceptional isomorphism $\PGL_2(\F_5) \cong S_5$,
  we see that our calculation also yields a Galois cover realizing a projective
  linear group.
\end{exm}

\section{Field of moduli and field of definition} \label{sec:fomfod}

Considering Grothendieck's original motivation for studying dessins, it is
important to consider the rather delicate issue of fields of definition of
\Belyi\ maps. In fact this is not only an engaging question on a theoretic
level, but it is also interesting from a practical point of view.  Indeed, as we
have seen in our calculations above, it is often necessary to determine
equations for \Belyi\ maps by recognizing complex numbers as algebraic numbers.
A bound on the degree $K$ is an important part of the input to the LLL algorithm
that is typically used for this. Moreover, having an estimate for the degree of
$K$ is a good indication of how computable a given cover will be---if the
estimate for the size is enormous, we are very unlikely to succeed in practice!

\subsection*{Field of moduli}

For a curve $X$
defined over $\Qbar$, the \defi{field of moduli} $M(X)$ of $X$ is the fixed
field of the group $\{\tau \in \Gal(\Qbar / \Q) : X^{\tau} \cong X\}$ on
$\Qbar$, where as before $X^{\tau}$ is the base change of $X$ by the
automorphism $\tau \in \Gal(\Qbar / \Q)$ (obtained by applying the automorphism
$\tau$ to the defining equations of an algebraic model of $X$ over $\Qbar$). One
similarly defines the field of moduli of a \Belyi\ map: $M(f)$ is the fixed
field of $\{\tau \in \Gal(\Qbar / \Q) : f^{\tau} \cong f\}$ with isomorphisms as
defined in Section 1. 

Now let $f:X \to \PP^1$ be a \Belyi\ map with monodromy representation $\sigma :
F_2 \to S_d$ and monodromy group $G$. By Theorem \ref{thm:galinv}, the monodromy
group $G$ of $f$, considered as a conjugacy class of subgroups of $S_d$, is
invariant under the Galois action. Therefore, given $\tau \in \Gal(\Qbar / \Q)$,
the conjugated morphism $f^\tau:X^\tau \to \PP^1$ is a \Belyi\ map, and its
monodromy representation $\sigma^\tau :F_2 \to S_d$, which is well-defined up to
conjugation, can be taken to have image $G$.  Because the Galois action
preserves the monodromy group up to conjugation and the ramification indices
\cite{JonesStreit},  the \Belyi\ map $f^{\tau}$ has the same passport $P$ as
$f$. We therefore get an action of $\Gal (\Qbar / \Q)$ on the set $S$ of \Belyi\ maps
with passport $P$. Since the stabilizer of an element of $S$ under this action
has index at most $\# S$ in $\Gal (\Qbar / \Q)$, we get the following result.

\begin{prop}\label{prop:merebound}
  Let $f$ be a \Belyi\ map with passport $P$ and field of moduli $K$. Then the
  degree $[K : \Q]$ is bounded above by the size of $P$.
\end{prop}

As mentioned at the end of Section \ref{sec:backg}, finding better bounds than
in Proposition \ref{prop:merebound} is far from trivial and a subject of ongoing
research.  Experimentally, the bound is often an equality for generic
(non-Galois) \Belyi\ maps.

\subsection*{Rigidified categories}

Working with Galois \Belyi\ maps and the additional structure coming from their
automorphism group naturally leads one to consider a new, more rigidified
category \cite{DebesDouai}. A \defi{$G$-\Belyi\ map} is a pair $(f,i)$, where $f
: X \to \PP^1$ is a Galois \Belyi\ map and $i : G \xrightarrow{\sim} \Mon (f)$
is an isomorphism of the monodromy group of $f$ with $G$. A \defi{morphism} of
$G$-\Belyi\ maps from $(f,i)$ to $(f',i')$ is an isomorphism of \Belyi\ maps
$h:X \xrightarrow{\sim} X'$ that identifies $i$ with $i'$, i.e., such that 
\begin{equation} \label{eqn:higx}
h( i(g) x) = i'(g) h(x) \text{ for all $g \in G$ and $x \in X$}.
\end{equation}

A \defi{$G$-permutation triple} is a triple of permutations $(\sigma_0 ,
\sigma_1 , \sigma_{\infty})$ in $G$ such that $\sigma_0 \sigma_1 \sigma_{\infty}
= 1$ and such that $\sigma_0 , \sigma_1 , \sigma_{\infty}$ generate $G$. A
\defi{morphism} of $G$-permutation triples is an isomorphism of permutation
triples induced by simultaneous conjugation by an element in $G$.  The main
equivalence is now as follows.

\begin{prop}\label{prop:cateqG}
  The following categories are equivalent:
  \begin{enumroman}
    \item $G$-\Belyi\ maps;
    \item $G$-permutation triples;
    \item surjective homomorphisms $F_2 \to G$.
  \end{enumroman}
\end{prop}

We leave it to the reader to similarly rigidify the notion of dessins; it will
not be needed in what follows.

We will need a slight weakening of this notion in the following section. A
\defi{weak $G$-\Belyi\ map} is a pair $(f,i)$, where $f : X \to \PP^1$ is a
Galois \Belyi\ map and $i : H \hookrightarrow \Mon (f)$ is an isomorphism of the
monodromy group of $f$ with a subgroup $H$ of $G$. A \defi{morphism} of weak
$G$-\Belyi\ maps from $(f,i)$ to $(f',i')$ is an isomorphism of \Belyi\ maps $h:
X \xrightarrow{\sim} X'$ such that \eqref{eqn:higx} holds up to conjugation,
i.e., such that there exists a $t \in G$ such that $h(i(g) x)= i' (t g t^{-1})
h(x)$ for all $g \in G$ and $x \in X$.

A \defi{weak $G$-permutation triple} is a triple of permutations $(\sigma_0 ,
\sigma_1 , \sigma_{\infty})$ in $G$ such that $\sigma_0 \sigma_1 \sigma_{\infty}
= 1$. A \defi{morphism} of weak $G$-permutation triples is an isomorphism of
permutation triples induced by simultaneous conjugation by an element in $G$.
The equivalence of Proposition \ref{prop:cateqG} now generalizes to the
following result.

\begin{prop}\label{prop:cateqGweak}
  The following categories are equivalent:
  \begin{enumroman}
    \item weak $G$-\Belyi\ maps;
    \item weak $G$-permutation triples;
    \item homomorphisms $F_2 \to G$.
  \end{enumroman}
\end{prop}

The set of \Belyi\ maps of degree $d$ can be identified with the set of weak
$S_d$-\Belyi\ maps. In particular, whereas $G$-\Belyi\ maps are always
connected, weak $G$-\Belyi\ maps need not be.  The absolute Galois group
$\Gal(\Qbar/\Q)$ acts on the set of (weak) $G$-\Belyi\ maps, so we can again
define the field of moduli of these rigidified \Belyi\ maps.

Having introduced weak $G$-\Belyi\ maps, it makes sense to consider passports up
to the action of the monodromy group $G \subset S_d$ instead of the full group
$S_d$. We accordingly define the \defi{refined passport} of a (not necessarily
Galois) \Belyi\ map $f:X \to \PP^1$ to be the triple $(g,G,C)$, where $g$ is the
genus of $X$, the group $G$ is the monodromy group of $f$, and $C = (C_0 , C_1 ,
C_{\infty})$ are the conjugacy classes of $\sigma_0,\sigma_1,\sigma_{\infty}$ in
$G$, not the conjugacy classes in $S_d$ included in the (usual) passport. 

Fried \cite{Fried} shows how the conjugacy classes $C_i$ change under the Galois
action. Let $\sigma \in \Gal (\Qbar / \Q)$, let $n = \#G$, and let $\zeta_n
\in \Qbar$ be a primitive $n$-th root of unity. Then $\sigma$ sends $\zeta_n$ to
$\zeta_n^a$ for some $a \in (\Z / n \Z)^{\times}$. We obtain new 
conjugacy classes $C_i^{a}$ by raising a representative of
$C_i$ to the $a$th power. Then for any character $\chi$ of $G$ we have
\begin{align}
  \sigma ( \chi (C_i) ) = \chi (C_i^{a})  .
\end{align}
Let $\Q (\chi (C_i))$ be the field generated by the character values of the
conjugacy classes $C_i$. We have $\Q (\chi (C_i)) = \Q$ if and only if all conjugacy
classes of $G$ are rational, as for instance in the case $G = S_d$.  

\begin{prop}
  Let $(f,i)$ be a weak $G$-\Belyi\ map with refined passport $R$ and field of
  moduli $K$ as a weak $G$-\Belyi\ map. Then the degree $[K : \Q (\chi (C_i))]$
  is bounded above by the size of $R$.
\end{prop}

Calculating in the category of weak $G$-\Belyi\ maps can be useful even when
considering \Belyi\ maps without this additional structure. More precisely, this
is useful when using formulas that approximate the size of a passport. To this
end, let $G$ be a finite group and let $C_0 , C_1 , C_{\infty}$ be conjugacy
classes in $G$. Let $S$ be the set of isomorphism classes of weak $G$-\Belyi\
maps $\sigma = (\sigma_0 , \sigma_1 , \sigma_{\infty})$ with the property that
$\sigma_i \in C_i$ for $i \in \left\{ 0,1,\infty \right\}$. Then a formula that
goes back to Frobenius (see Serre \cite[Theorem 7.2.1]{Serre}) shows that
\begin{align}\label{eq:serre}
  \sum_{(f,i) \in S} \frac{1}{\Aut_G (f,i)} = \frac{\#C_0 \#C_1 
  \# C_{\infty}}{(\#G)^2} \sum_{\chi} \frac{\chi (C_0) \chi (C_1) \chi
  (C_{\infty})}{\chi (1)}.
\end{align}
Here the automorphism group $\Aut_G (f,i)$ is the group of automorphisms of
$(f,i)$ as a weak $G$-\Belyi\ map. The sum on the left of (\ref{eq:serre}) runs
over all weak \Belyi\ maps with the aforementioned property; in particular, one
often obtains non-transitive solutions that one does not care about in practice.

When working with mere \Belyi\ maps (without rigidification as a weak
$G$-\Belyi\ map), it can still be useful to consider the estimate
(\ref{eq:serre}) when the monodromy group of the \Belyi\ map is question is
included in $G$.  We illustrate this by a few concrete examples.

\begin{exm}\label{ex:serreest}
  We start by taking $G$ to be a full symmetric group and give the
  above-mentioned estimate for the number of genus $0$ \Belyi\ maps with
  ramification passport 
  \begin{align*}
    (0, (3^2 2^3 , 5^1 4^1 2^1 1^1 , 6^1 4^1 2^1 )),
  \end{align*}
  Before giving it, we calculate the possible permutation triples up to
  conjugacy directly using Lemma \ref{lem:doubcos}. This shows that the number
  of solutions is $583$, of which $560$ are transitive. The transitive solutions
  all have monodromy group $S_{23}$ and hence trivial automorphism group. On the
  other hand, the Serre estimate (\ref{eq:serre}) equals $567 \frac{1}{4}$,
  which more precisely decomposes as
  \begin{align*}
    567 \frac{1}{4} = \frac{560}{1} + \frac{1}{1} + \frac{3}{2} +
    \frac{19}{4};
  \end{align*}
  of the 23 nontransitive solutions, there is only one with trivial automorphism
  group, whereas there are $3$ (resp.\ $19$) with automorphism group of
  cardinality $2$ (resp.\ $4$).  For each of the nontransitive solutions, the
  associated \Belyi\ maps are disjoint unions of curves of genus $1$, such as
  those corresponding to the products of the genus $1$ \Belyi\ maps with
  ramification types $(2^3,5^1 1^1,6^1)$ (which always have trivial automorphism
  group) with those with ramification types $(3^2,4^1 2^1 ,4^1 2^1 )$ (which
  have either $1$ or $2$ automorphisms, depending on the solution).
\end{exm}

\begin{exm}\label{ex:mathest}
  Another example is the case $(0, H, (4^4 2^2 1^3 , 4^4 2^2 1^3 , 5^4 1^3))$
  with $H \leq M_{23}$. We can identify $M_{23}$-conjugacy classes with
  $S_{23}$-conjugacy classes for these groups, as the conjugacy classes of
  $S_{23}$ do not split upon passing to $M_{23}$.
  
  The calculations are much more rapid working with $M_{23}$ than for the full
  group $S_{23}$. We obtain the estimate $909$, which fortunately enough equals
  the exact number of solutions because the corresponding subgroups of $M_{23}$
  all have trivial centralizer; this is not the case when they are considered as
  subgroups of $S_{23}$. Of these many solutions, it turns out that only $104$
  are transitive.  

  As mentioned before, this estimate only gives the number of weak
  $M_{23}$-\Belyi\ maps; accordingly, permutation triples are only considered
  isomorphic if they are conjugated by an element of $M_{23}$ rather than
  $S_{23}$. However, since $M_{23}$ coincides with its own normalizer in
  $S_{23}$, this coincides with the number of solutions under the usual
  equivalence. Directly working with the group $M_{23}$ indeed saves a great
  deal of computational overhead in this case.
\end{exm}

An explicit (but complicated) formula, using M\"obius inversion to deal with the
disconnected \Belyi\ maps, was given by Mednykh \cite{Mednykh} this vein; in
fact, his formula can be used to count covers with specified ramification type
of an arbitrary Riemann surface.

Finally, we mention that in the Galois case, the situation sometimes simplifies:
there are criteria \cite{ConderJonesStreitWolfart, JonesStreitWolfart} for
Galois \Belyi\ maps to have cyclotomic fields of moduli, in which case the
Galois action is described by a simple powering process known as \defi{Wilson's
operations}. Additionally, Streit--Wolfart \cite{StreitWolfart} have calculated
the field of moduli of an infinite family of \Belyi\ maps whose Galois group is
a semidirect product $\Z_p \rtimes \Z_q$ of cyclic groups of prime order.

\subsection*{Field of moduli versus field of definition}

We have seen that in all the categories of objects over $\Qbar$ considered so
far (curves, \Belyi\ maps, etc.)\ there is a field of moduli for the action of
$\Gal(\Qbar/\Q)$.  Given an object $Y$ of such a category with field of moduli
$M$, it is reasonable to ask whether $Y$ is \defi{defined over $M$}, i.e., if
there exists an object $Y_M$ in the appropriate category over $M$ that is
isomorphic with $Y$ over $\Qbar$, in which case $M$ is said to be a \defi{field
of definition} of $Y$. For example, if $Y = (X,f)$ is a \Belyi\ map over
$\Qbar$, this means that there should exist a curve $X_M$ over $M$ and a \Belyi\
map $f_M : X_M \rightarrow \PP_M^1$ such that $(X,f)$ can be obtained from
$(X_M,f_M)$ by extending scalars to $\Qbar$.

We first consider the case of curves. Curves of genus at most $1$ are defined
over their field of moduli. But this ceases to be the case for curves of larger
genus in general, as was already observed by Earle \cite{Earle} and Shimura
\cite{Shimura}. The same is true for \Belyi\ maps and $G$-\Belyi\ maps. This
issue is a delicate one, and for more information, we refer to work of
Coombes--Harbater \cite{CoombesHarbater}, D\`{e}bes--Ensalem
\cite{DebesEmsalem}, D\`{e}bes--Douai, \cite{DebesDouai}, and K\"ock
\cite{Kock}.

The obstruction can be characterized as a lack of rigidification. For example, a
curve furnished with an embedding into projective space is trivially defined
over its field of moduli (as a projectively embedded curve).  Additionally,
marking a point on the source $X$ of a \Belyi\ map and passing to the
appropriate category, \cite[Theorem 2]{Birch} states that the field of moduli is
a field of definition for the \defi{pointed \Belyi\ curve} \cite[Theorem
2]{Birch}; however, this issue seems quite subtle, and in \cite{Kock} only
auxiliary points with trivial stabilizer in $\Aut (X)$ are used. Note that the
more inclusive version of this rigidification (with possibly non-trivial
stabilizer) was considered in Section \ref{sec:mod} (see e.g.\
\eqref{eq:G2corresp}).  As mentioned at the beginning of the previous
subsection, this implication can then be applied to give an upper bound on the
degree of the field of definition of a \Belyi\ map, an important bit of
information needed when for example applying LLL to recognize coefficients
algebraically.

Note that for a \Belyi\ map $f : X \rightarrow \PP^1$, the curve $X$ may descend
to its field of moduli in the category of curves while $f$ does not descend to
this same field of moduli in the category of \Belyi\ maps. Indeed, this can be
seen already for the example $X = \PP^1$, as $\Gal(\Qbar/\Q)$ acts faithfully on
the set of genus $0$ dessins.  In general, the problem requires careful
consideration of obstructions that lie in certain Galois cohomology groups
\cite{DebesDouai}.

\begin{rmk}
  Although in general we will have to contend with arbitrarily delicate
  automorphism groups, Couveignes \cite{CouveignesAPropos} proved that every
  curve defined over a number field $K$ admits a \Belyi\ map without
  automorphisms defined over $K$.  This map will then necessarily not be
  isomorphic to any of its proper conjugates.
\end{rmk}

On top of all this, a \Belyi\ map may descend to its field of moduli \defi{in
the weak sense}, i.e., as a cover of a possibly non-trivial conic ramified above
a Galois-stable set of three points, rather than \defi{in the strong sense}, as
a cover of $\PP^1 \backslash \left\{ 0,1,\infty \right\}$ (i.e., in the category
of \Belyi\ maps over the field of moduli).  This distinction also measures the
descent obstruction for hyperelliptic curves, as in work of
Lercier--Ritzenthaler--Sijsling \cite{LRS}. For \Belyi\ maps, a deep study of
this problem in genus $0$, beyond the general theory, was undertaken by
Couveignes \cite[\S\S 4--7]{CouveignesCRF}: he shows that for the \defi{clean
trees}, those \Belyi\ maps with a single point over $\infty$ and only
ramification index $2$ over $1$, on the set of which $\Gal (\Qbar / \Q)$ acts
faithfully, the field of moduli is always a field of definition in the strong
sense. Moreover, he shows that in genus $0$, the field of moduli is always a
field of definition in the weak sense as long as the automorphism group of the
\Belyi\ map is not cyclic of even order, and in the strong sense as long as the
automorphism group is not cyclic.

These considerations have practical value in the context of computations. For
example, Couveignes \cite[\S 10]{CouveignesCRF} first exhibits a genus $0$
\Belyi\ map that descends explicitly to $\Q$ in the strong sense.  Then, due to
the presence of non-trivial automorphisms of this \Belyi\ map, one can realize
it as a morphism $f : C \to \PP^1$ for infinitely many mutually non-isomorphic
conics $C$ over $\Q$. And by choosing $C$ appropriately (not isomorphic to
$\PP^1$ over $\Q$), Couveignes manages to condense his equations from half a
page to a few lines.  Further simplification techniques will be considered in
the next section.

We mention some results on the field of moduli as a field of definition that are
most useful for generic ($G$-)\Belyi\ maps.
\begin{enumerate}
  \item If a curve or ($G$-)\Belyi\ map has trivial automorphism group, then it
    can be defined over its field of moduli, by Weil's criterion for descent
    \cite{Weil}.
  \item If the center of the monodromy group of a Galois \Belyi\ map is trivial, then
    it can be defined over its field of moduli by the main result in the article
    by D\`ebes--Douai \cite{DebesDouai}.
  \item A $G$-\Belyi\ map, when considered in the category of \Belyi\ maps
    (without extra structure) is defined over its field of moduli as a \Belyi\
    map \cite{CoombesHarbater}.
\end{enumerate}

To give an impression of the subtleties involved, we further elaborate on
Example \ref{exm:gal1} from the previous section. Along the way, we will
illustrate some of the subtleties that arise when considering fields of moduli.
As we will see, these subtleties correspond with very natural questions on the
level of computation.

\begin{exm}
  Since $\Delta(3,5,5)$ is a subgroup of $\Delta(2,5,6)$ of index $2$, we also obtain
  from this example a \Belyi\ map with indices $3,5,5$ for the group $A_5$ by
  taking the corresponding quotient.  Indeed, ramification can only occur over
  the points of order $2$ and $6$, which means that in fact the cover is a
  cyclic degree $2$ map of conics ramifying of order $2$ over these points and
  under which the point of order $5$ has two preimages.  An equation for this
  cover (which is a \Belyi\ map) can now be found by drawing an appropriate
  square root of the function $(s_3^5 / s_5^3) + (15/2)^2$ (which indeed
  ramifies of order $6$ over $\infty$ and of order $2$ over $0$) and sending the
  resulting preimages $\pm 15/2$ of the point of order $5$ to $0$ and $1$,
  respectively.

  Alternatively, we can calculate as follows.  The full ring of invariant
  homogeneous polynomials for $A_5$ (acting linearly by permutation of
  coordinates) is generated by the power sums
  $p_1 , \dots , p_5$ and the \defi{Vandermonde polynomial}
  \begin{align*}
    a = (v - w) (v - x) (v - y) (v - z) (w - x) (w - y) (w - z) (x - y) (x - z)
    (y - z)  .
  \end{align*}
  One easily determines the expression for $a^2$ in terms of the $p_i$; setting
  $p_1 = p_2 = p_4 = 0$, we get the relation
  \begin{align*}
    a^2 = \frac{4}{45} s_3^5 s_5 + 5 s_5^3  .
  \end{align*}
  This suggests that to get a function realizing the quotient $X \to X / A_5$,
  we take the map $g : X \rightarrow C$, where $C$ is the conic
  \begin{align*}
    C : 45y^2 = 4x z + 225 z^2
  \end{align*}
  and $g$ is given by
  \begin{align*}
    g(v,w,x,y,z) = (s_3^5 : a s_5 : s_5^3)  .
  \end{align*}
  Note that $Q$ admits the rational point $(1:0:0)$.

  This result is not as strong as one would like. As we have seen when
  calculating the full quotient $f$, the branch points of $g$ of order $5$ on
  $C$ satisfy $(x : z) = (0 : 1)$. But the corresponding points are only defined
  over $\Q (\sqrt{5})$, so this is a descent of a \Belyi\ map in the weak sense.
  We explain at the group-theoretical level what other kinds of descent can be
  expected.

  There are actually two Galois covers with ramification indices $(3,5,5)$ for
  $A_5$ up to isomorphism. The other cover is not found as a subcover of $f$;
  when composing with the same quadratic map, we instead get a Galois \Belyi\
  map whose Galois group is the direct product of $A_5$ and $\Z / 2 \Z$. The
  corresponding curve is given by taking the hyperelliptic cover ramified over
  the vertices of an icosahedron, leading to the equation
  \begin{align*}
    t^2 = s^{20} + 228 s^{15} + 494 s^{10} - 228 s^5 + 1  .
  \end{align*}

  In particular, this means that the Galois orbit of these covers consists of a
  single isomorphism class, as their monodromy groups upon composition differ
  \cite{Wood}. As mentioned above, an $A_5$-\Belyi\ map, considered as a mere
  \Belyi\ map, is defined over its field of moduli as a \Belyi\ map, so our
  equations above can be twisted to a \Belyi\ map over $\Q$, that is, with
  ramification at three rational points.

  However, the Galois cover does \emph{not} descend as an $A_5$-\Belyi\ map (so
  in the strong sense, as a Galois cover unramified outside $\left\{ 0,1,\infty
  \right\}$). Indeed, the character table of $A_5$ is only defined over $\Q
  (\sqrt{5})$.  Twisting may therefore give a cover defined over $\Q$, but the
  Galois action will then only be defined over $\Q (\sqrt{5})$ and be
  accordingly more complicated. We therefore forgo this calculation and content
  ourselves with the symmetric form above.
\end{exm}

For more on the questions considered in this section, see also further work by
Couveignes \cite{CouveignesQuelques}, and in a similar vein, the work of van
Hoeij--Vidunas on covers of conics \cite[\S\S 3.3--3.4]{vHV}, \cite[\S 4]{vHV2}.
We again refer to the fundamental paper of D\`ebes--Douai \cite{DebesDouai}, in
which strong results are given for both \Belyi\ maps and $G$-\Belyi\ maps that
suffice in many concrete situations.  Admittedly, this subject is a delicate one,
and we hope that computations will help to further clarify these nuances.

\section{Simplification and verification}\label{sec:veri} 

Once a potential model for a \Belyi\ map has been computed, it often remains to
simplify the model as much as possible and to verify its correctness
(independently of the method used to compute it). The former problem is still
open in general; the latter has been solved to a satisfactory extent.

\subsection*{Simplification}

By \defi{simplifying} a \Belyi\ map $f : X \to \PP^1$, we mean to reduce the
total (bit) size of the model.  Lacking a general method for doing this, we
focus on the following:
\begin{enumerate}
  \item If $X$ is of genus $0$, we mean to find a coordinate on $X$ that decreases
    the (bit) size of the defining coefficients of $f$.
  \item If $X$ is of strictly positive genus, we mean to simplify the defining
    equations for $X$. (In practice, this will also lead to simpler coefficients
    of the \Belyi\ map $f$.)
\end{enumerate}

Problem (1) was considered by van Hoeij--Vidunas \cite[\S 4.2]{vHV} under
the hypotheses that one of the ramification points has a minimal polynomial of
degree at most $4$; one tries to find a smaller polynomial defining the
associated number field and changes the coordinate accordingly, which typically
yields one a simpler expression of the \Belyi\ map.

Problem (1) is directly related with Problem (2) for hyperelliptic curves, since
simplifying the equations for hyperelliptic curves over a field $K$ boils down
to finding a small representative of the $\GL_2 (K)$-orbit of a binary form.
Typically one also requires the defining equation to have integral coefficients.
For the case $K = \Q$, this leads one to consider the problem of finding simpler
representations for binary forms under the action of the group of integral
matrices $\SL_2 (\Z)$.  This is considered by Cremona--Stoll
\cite{CremonaStoll}, using results from Julia \cite{Julia} to find a binary
quadratic covariant, to which classical reduction algorithms are then applied.
The resulting algorithms substantially reduce the height of the coefficients of
the binary form in practice, typically at least halving the bit size of already
good approximations in the applications \cite{CremonaStoll}. A generalization
to, and implementation for, totally real fields is given in Bouyer--Streng
\cite{BouyerStreng}. 

In fact, corresponding results for the simplification of \Belyi\ maps can be
obtained by taking the binary form to be the product of the numerator and
denominator of the \Belyi\ map. That the resulting binary form may have double
roots and hence may not correspond to hyperelliptic curves is no problem; see
the discussion by Cremona--Stoll \cite[after Proposition 4.5]{CremonaStoll}.

This problem of reduction is intimately related with the problem of finding a
good model of a \Belyi\ map or hyperelliptic curve over $\Z$. Note that even for
the case $K = \Q$ we have not yet used the full group $\GL_2 (\Q)$; the
transformations in $\SL_2 (\Z)$ considered by Cremona and Stoll preserve the
discriminant, but it could be possible that a suitable rational transformation
decreases this quantity while still preserving integrality of the binary form.
An approach to this problem is given by Bouyer--Streng \cite[\S
3.3]{BouyerStreng}.

In general, Problem (2) is much harder, if only because curves of high genus
become more difficult to write down. 

\begin{ques}
  Are there general methods to simplifying equations of curves defined over a
  number field in practice?
\end{ques}

\subsection*{Verification}

Let $f:X \to \PP^1$ be a map defined over a number field $K$ of degree $d$ that
we suspect to be a \Belyi\ map of monodromy group $G$, or more precisely to
correspond to a given permutation triple $\sigma$ or a given dessin $D$. To show
that this is in fact the case, we have to verify that
\begin{enumroman}
  \item $f$ is indeed a \Belyi\ map;
  \item $f$ has monodromy group $G$; and
  \item $f$ (or its monodromy representation) corresponds to the permutation
    triple $\sigma$; or
  \item[(iii)'] the pullback under $f$ of the closed interval $[0,1]$ is
    isomorphic (as a dessin) to $D$.
\end{enumroman}
This verification step is necessary for all known methods, and especially when
using the direct method from Section \ref{sec:groeb}; the presence of parasitic
solutions means that not even all solutions of the corresponding system of
equations will be \Belyi\ maps, let alone \Belyi\ maps with correct monodromy
group or pullback.

Point (i) can be computationally expensive, but it can be accomplished, by
using the methods of computational algebraic geometry.  Not even if $X = \PP^1$ is this
point trivial, since although verifying that a \Belyi\ map is returned is
easy for dessins of small degree, we need better methods than direct
factorization of the polynomials involved as the degree mounts.

As for point (ii), one simple sanity check is to take a field of definition $K$
for $f$ and then to substitute different $K$-rational values of $t \not\in
\{0,1,\infty\}$. One obtains an algebra that is again an extension of $K$ of
degree $d$ and whose Galois group $H$ must be a subgroup of the monodromy group
$G$ by an elementary specialization argument.  So if we are given a finite
number of covers, only one of which has the desired monodromy group $G$, then to
eliminate a cover in the given list it suffices to show that specializing this
cover gives a set of cycle type in $H$ that is not contained in the given
monodromy group $G$ when considered as a subgroup of $S_d$. Such cycle types can
be obtained by factoring the polynomial modulo a small prime of $K$. 

There are many methods to compute Galois groups effectively in this way; a
general method is given by Fieker--Kl\"uners \cite{FiekerKlueners}. This method
proceeds by computing the maximal subgroups of $S_d$ and checking if the Galois
group lies in one of these subgroups by evaluating explicit invariants.  This
method works well if $G$ has small index in $S_d$.  Iterating, this allows on to
compute the monodromy group of a \Belyi\ map explicitly instead of merely giving
the maximal groups in which is it included. To this end, one may work modulo a
prime $\frakp$ of good reduction, and in light of Beckmann's Theorem, we may
still reasonably expect a small prime of the ring of integers of $K$ that is
coprime with the cardinality $\#G$ of the monodromy group to do the job.

Second, one can compute the monodromy by using numerical approximation.  This
has been implemented by van Hoeij \cite{vanHoeij}, though one must be very
careful to do this with rigorous error bounds.  This idea was used by Granboulan
\cite{Granboulan} in the computation of a cover with Galois group $M_{24}$,
first realized (without explicit equation) by Malle--Matzat
\cite[III.7.5]{MalleMatzat}.  In particular, Schneps \cite[\S III.1]{Schneps}
describes a numerical method to draw the dessin itself, from which one can read
off the mondromy.  This method is further developed by Bartholdi--Buff--von
Bothmer--Kr\"oker \cite{Bartholdi}, who lift a Delaunay triangulation
numerically and read off the permutations by traversing the sequence of edges
counterclockwise around a basepoint. In particular this solves (iii): if we
express each of the complex solutions obtained by embedding $K \hookrightarrow
\C$, we may also want to know which cover corresponds to which permutation
triple up to conjugation.

A third and final method is due to Elkies \cite{ElkiesM23}, who uses an
effective version (due to Weil's proof of the Riemann hypothesis for curves over
finite fields) of the Chebotarev density theorem in the function field setting.
This was applied to distinguish whether the Galois group of a given cover was
equal to $M_{23}$ or $A_{23}$. More precisely, one relies on reduction modulo a
prime whose residue field is prime of sufficiently large characteristic (in his
case, $>10^9$) and uses the resulting distribution of cycle structures to deduce
that the cover was actually $M_{23}$.  This method has the advantage of using
exact arithmetic and seems particularly well-suited to verify monodromy of large
index in $S_d$.

\section{Further topics and generalizations}\label{sec:gens}

This section discuss some subjects that are generalizations of or otherwise
closely related with \Belyi\ maps. At the end, we briefly discuss the theoretical
complexity of calculating \Belyi\ maps.

\subsection*{Generalizations}

Over $\overline{\F}_p$, one can consider the reduction of \Belyi\ maps from
characteristic $0$; this is considered in Section \ref{sec:padic} above.
Switching instead to global function fields might be interesting, especially if
one restricts to tame ramification and compares with the situation in
characteristic $0$.  As a generalization of \Belyi's theorem, over a perfect
field of characteristic $p>0$, every curve $X$ has a map to $\PP^1$ that is
ramified only at $\infty$ by work of Katz \cite{Katz}. But this map is
necessarily wildly ramified at $\infty$ if $g(X) > 0$, so the corresponding
theory will differ essentially from that of \Belyi\ maps over $\Qbar$.

If we view \Belyi's theorem as the assertion that every curve over a number
field is an \'etale cover of $\PP^1 \setminus \{0,1,\infty\} \cong \calM_{0,4}$,
the moduli space of genus $0$ curves with $4$ marked points, then \Belyi's
result generalizes to a question by Braungardt \cite{Braungardt}: is every
connected, quasi-projective variety $X$ over $\Qbar$ birational to a finite
\'etale cover of some moduli space of curves $\calM_{g,n}$?  Easton and Vakil
also have proven that the absolute Galois group acts faithfully on the
irreducible components of the moduli space of surfaces \cite{EastonVakil}.
Surely some computations in small dimensions and degree will be just as
appealing as in the case of \Belyi\ maps.  

As mentioned on a naive level in Remarks \ref{rmk:naive1} and \ref{rmk:naive2},
another more general way to look at \Belyi\ maps is through the theory of
\defi{Hurwitz schemes}, which give a geometric structure to the set
$\calH_{n,r}(\Qbar)$ parametrizing degree $n$ morphisms to $\PP^1$ over $\Q$
that are ramified above $r$ points.  The theorem of \Belyi\ then amounts to
saying that by taking the curve associated to a morphism, one obtains a
surjective map from the union of the $\overline{\Q}$-rational points of the spaces $\calH_{n,3}$ to the union of the
$\overline{\Q}$-rational points of the moduli spaces of curves $\calM_g$ of
genus $g$. We refer to work of Romagny--Wewers \cite{RomagnyWewers} for a more
complete account. 

\subsection*{Origamis}

One generalization of \Belyi\ maps is given by covers called \defi{origamis}: covers
of elliptic curves that are unramified away from the origin.  For a more
complete account on origamis, see Herrlich--Schmith\"usen
\cite{HerrlichSchmithuesen}; \Belyi\ maps can be obtained from origamis by a
degeneration process \cite[\S 8]{HerrlichSchmithuesen}.

The reasons for considering origamis are many. First, the fundamental group of
an elliptic curve minus a point is analogous to that of the Riemann sphere, in
that it is again free on two generators. The ramification type above the origin
is now given by the image of the commutator of these two generators. The local
information at this single point of ramification reflects less information about
the cover than in the case of \Belyi\ maps. Additionally, the base curve can be
varied, which makes the subject more subtle, as Teichm\"uller theory makes its
appearance.

An exciting family of special origamis was considered by Anema--Top
\cite{AnemaTop}: they consider the elliptic curve $E : y^2 = x^3 + a x + b$ over
the scheme $B : 4 a^3 + 27 b^2 = 1$ defined by the constant non-vanishing
discriminant $1$ of $E$.  Considering the torsion subschemes $E[n]$ over $B$,
one obtains a family of covers over the base elliptic curve $B$ of $j$-invariant
$0$ that is only ramified above the point at infinity and whose Galois groups
are subgroups of special linear groups. It would be very interesting to deform
this family to treat the case of arbitrary base curves, though it is not clear
how to achieve this.

\begin{ques}
  How does one explicitly deform special origamis to families with arbitrary
  base curves?
\end{ques}

Explicit examples of actual families of origamis were found by
Rubinstein-Salzedo \cite{RS1,RS2}. In particular, by using a deformation
argument starting from a nodal cubic, he obtains a family of hyperelliptic
origamis that are totally ramified at the origin. For the case of degree $3$,
this gives a unique cover of genus $2$. More precisely, starting with an
elliptic curve $E$ with full $2$-torsion in Legendre form
\begin{align*}
  y^2 = x (x - 1) (x - \lambda)  ,
\end{align*}
the hyperelliptic curve
\begin{align*}
  y^2 = \frac{1}{2}\left(-4 x^5 + 7 x^3 - (2 \lambda - 1) x^2 - 3 x + (2 \lambda
  - 1) \right)  
\end{align*}
admits a morphism to $E$ given by
\begin{align*}
  \left( x,y\right) \longmapsto \left( \frac{1}{2} \left( -4 x^3 + 3 x + 1
  \right) , \frac{y}{2} \left(-4 x^2 + 1 \right)  \right)
\end{align*}
that is only ramified at the points at infinity of these curves.

It is important to note here that the field of moduli of these covers is an
extension of the field of moduli of the base elliptic curve; more precisely, as
suggested by the formulas above, this field of moduli is exactly the field
obtained by adjoint the $2$-torsion of the curve. This is a variation on a
result in Rubinstein-Salzedo \cite{RS1}, where simpler expressions for similar
covers are found in every degree. Amusingly enough, adjoining the full
$2$-torsion of the base curves always suffices to define these covers.  This
result is appealing and quite different from the corresponding situation for
\Belyi\ maps, and therefore we ask the following question.

\begin{ques}
  Which extension of the field of moduli is needed to define similar covers
  totally ramified above a single point for general curves?
\end{ques}

\subsection*{Specialization}

Covering maps of the projective line with more than 3 ramification points
specialize to \Belyi\ maps by having the ramification points coincide.  In many
cases, the covers in the original spaces are easier to compute, and this
limiting process will then lead to some non-trivial \Belyi\ maps.  This also
works in reverse, and provides another application of computing \Belyi\ maps.
Hallouin--Riboulet-Deyris \cite{HallouinModuli} explicitly computed
polynomials with Galois group $A_n$ and $S_n$ over $\Q(t)$ with four branch
points for small values of $n$; starting from a relatively simple ``degenerate''
three-point branched cover, the four-point branched cover is obtained by complex
approximation (using Puiseux expansions). These methods were considerably
augmented by Hallouin in \cite{HallouinStudy} to find another such family with
group $\PSL_2 (\F_8)$. More recently, K\"onig \cite{Konig} similarly computed
such an extension of $\Q (t)$ with Galois group $\PSL_5(\F_2)$, using a $p$-adic
approximation to calculate the initial three-point degeneration. In all
aforementioned cases, the resulting covers can be specialized to find explicit
solution to the inverse Galois problem for the groups involved, and as mentioned
at the end of Section \ref{sec:backg}, the results from \cite{HallouinStudy}
have also found an application in the determination of equations for Shimura
curves \cite{HallouinComputation}.


As mentioned in Section \ref{sec:coan}, Couveignes \cite{CouveignesTools} has
used a patching method to describe more generally the computation of families of
ramified branch covers, using a degeneration to the situation of three-point
covers. More extensive algorithmic methods to deal with this question should
therefore be in reach of the techniques of numerical algebraic geometry.

\subsection*{Complexity}

In this article, we have been primarily concerned with practical methods for
computing \Belyi\ maps; but we conclude this section by posing a question
concerning the theoretical complexity of this task.

\begin{ques}
  Is there an algorithm that takes as input a permutation triple and produces as
  output a model for the corresponding \Belyi\ map over $\Qbar$ that runs in
  time doubly exponential in the degree $n$?  
\end{ques}

There is an algorithm (without a bound on the running time) to accomplish this
task, but it is one that no one would ever implement: there are only countably
many \Belyi\ maps, so one can enumerate them one at a time in some order and use
any one of the methods to check if the cover has the desired monodromy.  It
seems feasible that the Gr\"obner method would provide an answer to the above
problem, but this remains an open question.  Javanpeykar \cite{Javanpeykar} has
given explicit bounds on the Faltings height of a curve in terms of the degree
of a \Belyi\ map; in principle, this could be used to compute the needed
precision to recover the equations over a number field.

\end{document}